\documentclass[11pt,rqno, a4paper]{article}

\usepackage{xcolor}
\pagecolor{white}

\usepackage{titlesec}
\titleformat*{\section}{\sc\centering\large} % INDICATES THE FONT SIZE OF THE SECTIONS. 
\titleformat*{\subsection}{\bf} % INDICATES THE FONT SIZE OF THE SUBSECTIONS. 
\titleformat*{\subsubsection}{\it} % INDICATES THE FONT OF SUBSECTIONS.

\usepackage{graphicx}%
\usepackage{multirow}%
\usepackage{amsmath,amssymb,amsfonts}%
\usepackage{amsthm}%
\usepackage{mathrsfs}%
\usepackage[title]{appendix}%
\usepackage{xcolor}%
\usepackage{textcomp}%
\usepackage{manyfoot}%
\usepackage{booktabs}%
\usepackage{algorithm}%
\usepackage{algorithmicx}%
\usepackage{algpseudocode}%
\usepackage{listings}%

\usepackage{mathtools}
\usepackage{multirow}
\usepackage{verbatim}
\usepackage{graphicx}
\usepackage[shortlabels]{enumitem}
\usepackage{hyperref}
\usepackage[noabbrev]{cleveref}
\usepackage{caption}
\usepackage[normalem]{ulem}
\usepackage{amsthm,amscd,upref,amstext}

\newcommand{\Sa}{{\mathbb{S}}^{1}}

\DeclarePairedDelimiterX{\norm}[1]{\lVert}{\rVert}{#1}
\newcommand{\abs}[1]{\left\lvert#1\right\rvert}

\newcommand{\T}{\mathbb{T}}
\newcommand{\R}{\mathbb{R}}

\usepackage{amsmath}
\usepackage{amssymb}
\usepackage{amsfonts}
\usepackage{setspace}
\usepackage{latexsym}
\usepackage{mathrsfs}
\usepackage{epsfig}
\usepackage{amsthm}
\usepackage{color}

\usepackage[margin=2cm]{caption}

\usepackage{comment}

\usepackage{transparent}
\graphicspath{{./}{images/}}

\usepackage{import}
\usepackage{xifthen}
\usepackage{pdfpages}
\usepackage{transparent}

\newcommand{\be}{\begin{equation}}
\newcommand{\ee}{\end{equation}}

\usepackage{fancyhdr}
\newtheorem*{rep@theorem}{\rep@title}
\newcommand{\newreptheorem}[2]{%
\newenvironment{rep#1}[1]{%
 \def\rep@title{#2 \ref{##1}}%
 \begin{rep@theorem}}%
 {\end{rep@theorem}}}
\makeatother

\newtheorem{theorem}{Theorem}[section]
\newtheorem{lemma}[theorem]{Lemma}

\newtheorem{corollary}[theorem]{Corollary}
\newtheorem{proposition}[theorem]{Proposition}

\newreptheorem{theorem}{Theorem}

\usepackage{tocloft} % TABLE OF CONTENTS CUSTOMIZATION

\usepackage{setspace, tocloft}

\allowdisplaybreaks

\usepackage[auth-lg]{authblk}
\begin{document}
\setcounter{Maxaffil}{1}
 \renewcommand\Affilfont{\itshape}
\title{Towards Long-time Geometrization: Stability of the Double-Cusp Spacetimes Under ${\rm T}^2$-Symmetry}
\date{}
\author{{\Large Alejandro Bellati}\thanks{abellati@math.princeton.edu}}
\affil{Princeton University, Department of Mathematics, \authorcr  Fine Hall, Washington Road, Princeton, NJ 08544, USA.}
\author{{\Large Martín Reiris}\thanks{mreiris@cmat.edu.uy}}
\affil{Universidad de la Rep\'ublica, Centro de Matemática, \authorcr  Iguá 4225, Montevideo, 11400, Uruguay.}
\maketitle

\begin{abstract}
Since the early years of General Relativity, understanding the long-time behavior of the cosmological solutions of Einstein's vacuum equations has been a fundamental yet challenging task. Solutions with global symmetries, or perturbations thereof, have been extensively studied and are reasonably understood. On the other hand, thanks to the work of Fischer-Moncrief and M. Anderson, it is known that there is a tight relation between the future evolution of solutions and the Thurston decomposition of the spatial 3-manifold. Consequently, cosmological spacetimes developing a future asymptotic symmetry should represent only a negligible part of a much larger yet unexplored solution landscape. In this work, we revisit a program initiated by the second named author, aimed at constructing a new type of cosmological solution first posed by M. Anderson, where (at the right scale) two hyperbolic manifolds with a cusp separate from each other through a thin torus neck. Specifically, we prove that the so-called double-cusp solution, which models the torus neck, is stable under small $\Sa \times \Sa$ - symmetry-preserving perturbations. The proof, which has interest on its own, reduces to proving the stability of a geodesic segment as a wave map into the hyperbolic plane and partially relates to the work of Sideris on wave maps and the work of Ringstr\"om on the future asymptotics of Gowdy spacetimes. Additionally, we also establish the future long-time existence for the wave map equations for almost all initial data. The proof of this relies primarily on pointwise estimates derived using the energy-momentum tensor.
\end{abstract}

\newpage

\tableofcontents

\section{Introduction}
%%%%%%%%%%

Since the early years of General Relativity, understanding the long-time behavior of the cosmological solutions of Einstein's equations has been a fundamental yet quite challenging task. Solutions with spatial symmetries, like the spatially homogeneous Bianchi models or the Gowdy $\mathbb{T}^{2}$-symmetric spacetimes, have been extensively studied over the decades and are reasonably well understood \cite{ringstrom2009cauchy}, \cite{gowdy1974vacuum}, \cite{ringstrom2009strong}. While all these models are very valuable and provide explicit examples of future dynamics they fall short when the goal is to describe the full set of possible future behaviors. In this work, we revisit a program initiated by the second named author, aimed at constructing a new type of cosmological solution first posed by M. Anderson. This solution exhibits a qualitative behavior that differs significantly from any other known model. As we will explain, such a solution would provide strong support to some ideas developed by Fischer-Moncrief and Anderson relating fundamentally the topology of the Cauchy 3-hypersurfaces to the dynamics of the cosmological solutions \cite{fischer2000reduced}, \cite{fischer2001reduced}, \cite{anderson2001long}, \cite{anderson2004cheeger}.
\vspace{.1cm}

Motivated by certain considerations on the Thurston geometrization conjecture, Anderson posed in \cite{anderson2004cheeger} (see the end of the paper) the problem of finding a cosmological solution where coarsely speaking, two hyperbolic 3-manifolds with a cusp\footnote{A cusp is a 3-manifold $(-\infty,0]\times \mathbb{T}^{2}$ with a hyperbolic metric of the form $g_{H}=dx^{2}+e^{2x}g_T$, with $g_T$ flat $x-$independent on $\mathbb{T}^2$. Cusps are discussed later in the article.} separate from each other through a thin torus neck. In Anderson's picture, hypersurfaces $\Sigma_{k}$ of mean curvature $k\in (-\infty,0)$ evolve in the expanding direction $k\uparrow 0$, but the geometry at each time $k$ is scaled so that the mean curvature of $\Sigma_{k}$ is $-3$. Under this scaling, the two hyperbolic pieces with their corresponding cusp should emerge over time, separating from each other along an increasingly thin torus neck that develops asymptotically a $\mathbb{T}^{2}$-symmetry. \Cref{fig:posedin} schematizes that behavior. This spacetime would comprise a new and non-trivial example of a cosmological solution of the vacuum Einstein equations, whose spatial geometry (at the mentioned scale) evolves towards the Thurston decomposition of its Cauchy hypersurface. Furthermore, for this solution, the Fischer-Moncrief's reduced volume $\nu(k):=(-k/3)^{3}Vol(\Sigma_{k})$, would decay towards its topological lower bound given by $(-\sigma/6)^{3/2}$, where $\sigma$ is the Yamabe invariant of the $\Sigma$'s, \cite{fischer2000reduced}, \cite{fischer2001reduced}.

\begin{figure}[h]
\centering
    \def\svgwidth{.9\columnwidth}
    %% Creator: Inkscape 1.2 (dc2aedaf03, 2022-05-15), www.inkscape.org
%% Accompanies image file '5new.pdf' (pdf, eps, ps)
\begingroup%
  \makeatletter%
  \providecommand\color[2][]{%
    \errmessage{(Inkscape) Color is used for the text in Inkscape, but the package 'color.sty' is not loaded}%
    \renewcommand\color[2][]{}%
  }%
  \providecommand\transparent[1]{%
    \errmessage{(Inkscape) Transparency is used (non-zero) for the text in Inkscape, but the package 'transparent.sty' is not loaded}%
    \renewcommand\transparent[1]{}%
  }%
  \providecommand\rotatebox[2]{#2}%
  \newcommand*\fsize{\dimexpr\f@size pt\relax}%
  \newcommand*\lineheight[1]{\fontsize{\fsize}{#1\fsize}\selectfont}%
  \ifx\svgwidth\undefined%
    \setlength{\unitlength}{533.32071378bp}%
    \ifx\svgscale\undefined%
      \relax%
    \else%
      \setlength{\unitlength}{\unitlength * \real{\svgscale}}%
    \fi%
  \else%
    \setlength{\unitlength}{\svgwidth}%
  \fi%
  \global\let\svgwidth\undefined%
  \global\let\svgscale\undefined%
  \makeatother%
  \begin{picture}(1,0.39942167)%
    \lineheight{1}%
    \setlength\tabcolsep{0pt}%
    \put(0,0){\includegraphics[width=\unitlength,page=1]{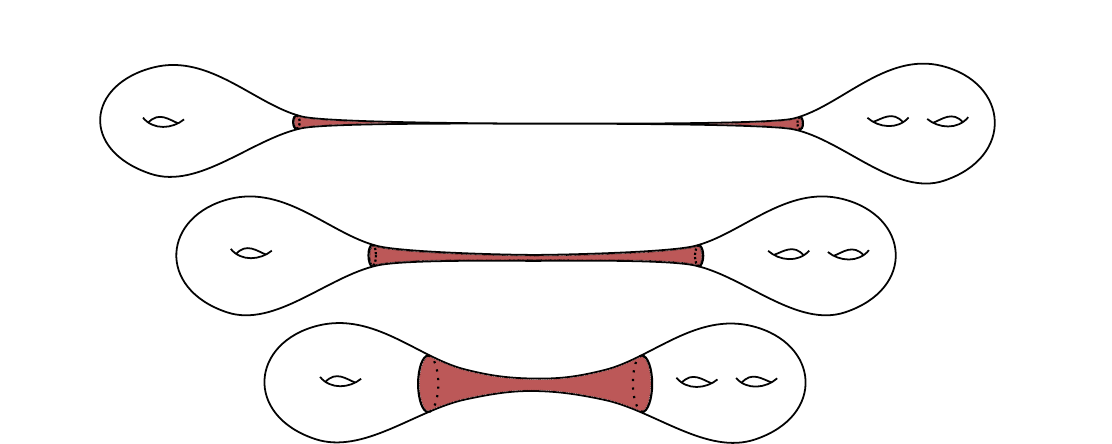}}%
    \put(0.40720955,0.33354263){\color[rgb]{0.69019608,0.23137255,0.23137255}\transparent{0.85000002}\makebox(0,0)[lt]{\lineheight{1.25}\smash{\begin{tabular}[t]{l}Torus neck\end{tabular}}}}%
    \put(-0.00368051,0.37092976){\color[rgb]{0,0,0}\transparent{0.85000002}\makebox(0,0)[lt]{\lineheight{1.25}\smash{\begin{tabular}[t]{l}Hyperbolic piece\end{tabular}}}}%
    \put(0.68820042,0.37092778){\color[rgb]{0,0,0}\transparent{0.85000002}\makebox(0,0)[lt]{\lineheight{1.25}\smash{\begin{tabular}[t]{l}Hyperbolic piece\end{tabular}}}}%
  \end{picture}%
\endgroup%

\caption{Expected behavior of the type of solution posed in \cite{anderson2004cheeger}.}
\label{fig:posedin}
\end{figure}
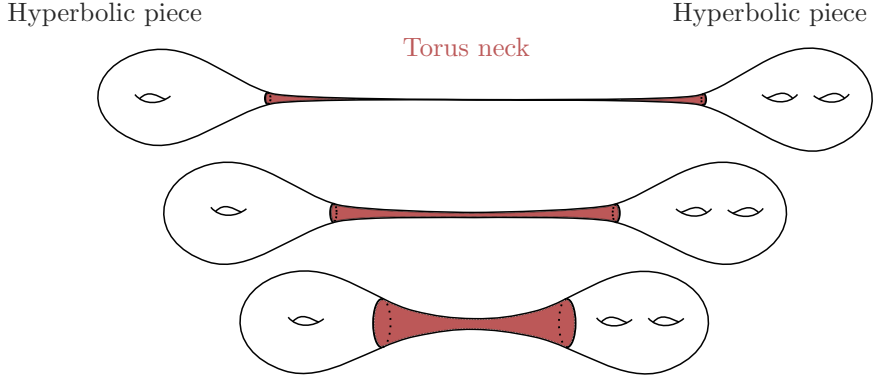

The double-cusps solutions were introduced in \cite{reiris2010ground} and are explicit $\mathbb{T}^{2}=\Sa\times\Sa$-symmetric solutions on $\mathbb{R}\times\mathbb{R}\times \mathbb{T}^{2}$ tailored to model the evolution of the torus neck. As we will see, they enjoy all the required global and asymptotic properties and are therefore a crucial piece in this context. However, before studying the combined evolution of the torus neck and the two hyperbolic manifolds with cusps, it is essential to establish that the double-cusp solutions are future-stable, and providing sharp decay estimates. This article is aimed mainly to prove the stability of these solutions under $\mathbb{T}^{2}$-symmetric perturbations. Interestingly, this problem is equivalent to proving the stability of a non-affinely parametrized geodesic segment (modeling the double cusp) as a wave map from a flat 3-dimensional spacetime into the hyperbolic plane. We note that the analysis of $\mathbb{T}^{2}$-symmetric perturbations is a natural first step inside the larger program, as non-symmetric degrees of freedom are expected to decay more rapidly than the symmetric ones. A detailed investigation of non-symmetric perturbations will be addressed in future work. Although the article primarily focuses on the stability of double cusps, we also establish a long-time existence theorem for this class of wave maps, applicable to nearly all initial data. This result is not directly needed for the stability proofs themselves; however, the $C^0$ estimate obtained during its derivation is useful for proving the non-polarized stability.     
\vspace{.2cm}

The stability problem of geodesics as wave maps was studied by Sideris in \cite{sideris1989global}, but the problem considered in that work is different from ours. Also, the future evolution of $\mathbb{T}^{3}$-Gowdy spacetimes was studied by Ringstr\"om in \cite{ringstrom2004wave}, through a wave map that, of course, has the same origin as ours. While in our case the wave map defines curves with the same ends as the geodesic segment, in \cite{ringstrom2004wave} it defines loops inside hyperbolic space. 
\vspace{0.2cm} 

In the next section we introduce the double cusp spacetimes and describe their global properties. The discussion there is not directly relevant for the proof of the main results, but helps to understand the geometric context and physical motivation. We also introduce in \cref{DCGP} the main equations that one needs for the rest of the paper.

\section{Double-cusps and their global properties}\label{DCGP}
%%%%%%%%%%%%%%%%%%%%%%%%%%%%%%%

\subsection{The evolution equations and the explicit expression of the double cusps.}
%%%%%%%%%%%%%%%%%%%%%%%%%%%%%%%%%%%%%%%%%

The double cusp spacetimes are vacuum $\mathbb{T}^{2}$-symmetric metrics with the following form \cite{reiris2010ground}, 
\begin{equation}\label{eq:metricform}
g = e^{2a}(-dt^2 + dx^2) + R(e^{2W} + q^2e^{-2W})d\theta_1^2 -2Rqe^{-2W}d\theta_1d\theta_2 + Re^{-2W}d\theta_2^2,
\end{equation}
over the manifold $\R_{t} \times \R_{x} \times \Sa_{\theta_{1}}\times \Sa_{\theta_{2}}$, where here $a$, $R$, $W$, and $q$ depend only on $t$ and $x$. The hypersurfaces $t = {\rm const}$, will be Cauchy hypersurfaces and are diffeomorphic to a ``neck'' $\mathbb{R}\times \mathbb{T}^{2}$. Metrics of this form are similar to the $\mathbb{T}^3$-Gowdy's metrics,  \cite{gowdy1974vacuum}, but differ from them in that $R$ is not taken as a coordinate and that $x$ is not periodic. Locally, the metric form (\ref{eq:metricform}) can be derived from the assumption of a $\mathbb{T}^2-$ free action by isometries on the Cauchy surfaces plus the vanishing of the twist constants (see \cite{gowdy1974vacuum}, \cite{ringstrom2004gowdy}, \cite{chrusciel1990space}). The double-cusps will be non-stationary spacetimes and this will add extra difficulties in the stability problem.

Before giving the explicit form of the double cusps, let us present the equations for $R$, $W$, $q$, and $a$ derived from the Einstein equations.  These equations will be used througouth the article as we will study the stability of the double cusps under metric perturbations maintaining the form \ref{eq:metricform}. They are,
\begin{gather}
R_{xx}-R_{tt} = 0  \label{eq:sysR},\\
W_{tt}-W_{xx} + \frac{R_t}{R}W_t - \frac{R_x}{R}W_x + \frac{(q_t^2-q_x^2)}{2}e^{-4W} = 0  \label{eq:sysW},\\
q_{tt}-q_{xx} + \frac{R_t}{R}q_t - \frac{R_x}{R}q_x -4q_tW_t + 4q_xW_x = 0 \label{eq:sysq},\\
a_{tt}-a_{xx} + \frac{R_x^2-R_t^2}{4R^2} + W_t^2-W_x^2 + \tfrac{1}{4}(q_t^2-q_x^2)e^{-4W} = 0 \label{eq:wave_a},
\end{gather}
and,
\begin{gather}
 a_t \frac{R_t}{R} + a_x\frac{R_x}{R} + \frac{1}{4}\left(\frac{R_x^2}{R^2} + \frac{R_t^2}{R^2}\right) - \frac{R_{xx}}{R} - (W_x^2 + W_t^2) - \frac{1}{4}e^{-4W}(q_x^2 + q_t^2) = 0 \label{eq:a_fuerte}, \\
 a_{x}\frac{R_t}{R} +a_{t}\frac{R_x}{R} - \frac{R_{tx}}{R} + \frac{R_xR_t}{2R^2} + 2W_tW_x - \frac{1}{2}e^{-4W}q_xq_t = 0 \label{eq:a_debil}.
\end{gather}
The equations (\ref{eq:sysR}), (\ref{eq:sysW}), (\ref{eq:sysq}) and (\ref{eq:wave_a}) are the dynamical equations for $R$, $W$, $q$ and $a$, and (\ref{eq:a_fuerte}) and (\ref{eq:a_debil}) are the constraint equations. The dynamical equation for $R$ decouples from all the others, and the dynamical equations for $W$ and $q$ decouple from that of $a$. In certain cases, one can solve globally for $a_x$ and $a_t$ from the \cref{eq:a_fuerte,eq:a_debil} and then simply perform line integrals to find $a$. In this case, $a$ is determined from $R, W$, and $q$ up to an integration constant.

It is crucial but also well known that the \cref{eq:sysW,eq:sysq} are wave map equations into the hyperbolic plane. This is seen as follows. Think of the hyperbolic plane $\mathbb{H}$ as $\mathbb{R}^{2}=\mathbb{R}_{x}\times \mathbb{R}_{y}$ endowed with the metric $h = 4dy^2 + e^{4y}dx^2$, consider the manifold $\R_{t} \times \R_{x} \times \Sa_{\phi}$ endowed with the metric $k = 4e^{4t}(-dt^2 + dx^2) + R^2(t,x)d\phi^2$ and denote this Lorentzian manifold as $\mathbb{K}$. Then $W$ and $q$ satisfy the \cref{eq:sysW,eq:sysq} if and only if the map $\chi:\mathbb{K}\longrightarrow \mathbb{H}$ given by
\begin{equation}
\chi(t,x,\phi) = (q(t,x),-W(t,x)),
\label{eq:chi}
\end{equation}
is a wave map between the two manifolds. Another way of expressing this is that \cref{eq:sysW,eq:sysq} are the Euler-Lagrange equations of the action
\begin{align}
S =  \int \partial_l \chi^i \partial_m \chi^j h_{ij}k^{lm} \, dV_{k} =  2\pi \int R(4(W_{x}^{2}+W_{t}^{2})+(q_{x}^{2}+q_{t}^{2})e^{-4W})dtdx.
\end{align}
Let us now see the explicit form of the double cusps. First, for all the double-cusp solutions, one takes $R(t,x)=R_0 e^{2t}\cosh 2x$ with $R_0$ a constant, which of course solves the \cref{eq:sysR}. Second, one requires $W$ and $q$ to be $t$-independent, i.e. $W=W(x)$ and $q=q(x)$. With this ansatz, the Euler-Lagrange equations are equivalent to the Euler-Lagrange equations of the action,
\begin{equation}
F=\int_{\mathbb{R}} |\gamma'|^{2}\cosh(2x)dx
\end{equation}
where $\gamma(x)=\chi(x)$, and whose solutions are well known to be non-affinely parametrized geodesic segments of the hyperbolic plane. When the geodesic segment is vertical and thus has $q$ constant, we say that the double cusp is {\it polarized}. Their explicit form is,
\begin{gather}\label{eq:background}
R = R_0 e^{2t}\cosh(2x),\\ \label{eq:background2}
W = W_1 + W_0 \arctan(e^{2x}),\\ \label{eq:background3}
q = q_{0},\\   \label{eq:background4}
a = a_{0} - \left(\frac{1}{2} + \frac{W_0^2}{2}\right)\frac{1}{2}\ln(\cosh(2x)) + \left(\frac{3}{2} + \frac{W_0^2}{2}\right)t,
\end{gather}
with $R_{0}>0, W_{0}\neq 0, W_{1}$, $q_{0}$ and $a_{0}$ constants. The non-polarized double-cusps are created by transforming polarized ones by an isometry of the hyperbolic plane (see Figure \ref{fig:asymptotic_conv}), and the explicit expression won't be particularly relevant. For this reason, without lost of generality, we will work with polarized double cusps solutions. From now on, we will be denoted this by $R_{b}, W_{b}, q_{b}$ and $a_{b}$, where `b' stands for `background'. 

The stability problem that we face amounts basically to the stability of $R_{b}$, $W_{b}$, $q_{b}$ and $a_{b}$ as particular solutions of a system of partial differential equations, and in fact this is pretty much the viewpoint that we will take. But there is a caveat here, as the stability, say in certain convenient norms, of them does not necessarily amount to the future stability of the spacetime. We will address this below and in the following section.

To study the evolution of $R, q, W$ and $a$, we will study first the wave equation for $R$, then the wave map equation for $(W,q)$, and finally, we will study $a$, which will be determined entirely from them. To control $R$ and of $(W, q)$, we will use some natural norms that are standard from a PDE point of view but that may not guarantee the stability of $a$, and thus of the spacetime as a whole as commented earlier. In the next section, we will see that this conflict can be solved just using finer norms to measure the smallness of the initial data for $R$, $W$, and $q$ over the Cauchy surface $t=0$, which implies a smallness of the initial data of $a$, sufficient to guarantee future geodesic completeness.  
%The double cusp spacetimes are stable for small perturbations in the sense of these finer norms, even though the evolution of $R, W$, and $q$ is controlled with more coarse ones. All these subtleties can be tracked down to the very nature of the coordinates $(t,x)$ and the nature of the Cauchy hypersurface $\{t=0\}$ where we are perturbing the initial data. 

In the following section, we review the global properties of the double cusps. This information will not play a role when studying the stability of $R$ and of $(W,q)$, which will be treated as a standard PDE problem, but it will help to understand the discussion about the stability of $a$ and therefore of the double cusp as a spacetime.    
\begin{figure}[!ht]
\centering
    \def\svgwidth{0.7\columnwidth}
    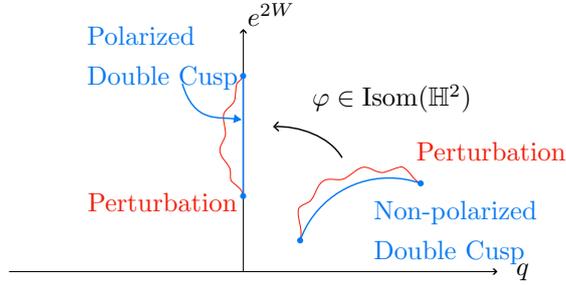

\caption{Two double cusps are represented by geodesics. A perturbation is represented as the curve $x \mapsto \chi(t,x)$ for a fixed $t$. As time evolves, this curve will move. The figure also illustrates how an arbitrary double cusp can be seen as a polarized double cusp.}
\label{fig:asymptotic_conv}
\end{figure}
\subsection{Global properties of the double cusps.}
%%%%%%%%%%%%%%%%%%%%%%%%%

A main property of the double cusps is that at each of their two ends, one can define spacetime coordinates $(t',x')$ and $(t'',x'')$ where one can observe the spatial scaled metrics converge towards hyperbolic cusps. This is one of the main properties making double cusps adequate to model the necks of the solutions posed in \cite{anderson2004cheeger}. To explain all that, we begin recalling certain notions on hyperbolic manifolds and flat cone spacetimes. If $(\mathbb{M},g_{H})$ is a hyperbolic manifold, then $\mathbb{R}_{\tau}\times \mathbb{M}$ endowed with the metric $g = -d\tau^{2}+\tau^{2}g_{H}$ is a flat spacetime (hence a solution of the Einstein equations) called a flat cone. The mean curvature of the hypersurface $\tau=\tau_{0}$ is $k_{0}=-3/\tau_{0}$. Therefore, when the spacetime metric $g$ is scaled as $(k_{0}/3)^{2}g = \tau_{0}^{-2}g=d(\tau/\tau_{0})^{2}+(\tau/\tau_{0})^{2}g_{H}$, then the mean curvature of the hypersurface $\tau=\tau_{0}$ becomes $-3$ and the induced 3-metric $g_{H}$. This is called CMC scaling and can be made at any CMC hypersurface $\Sigma_{k}$ of mean curvature $k$ inside a spacetime. Hyperbolic manifolds of finite volume can be non-compact. When this is so, the manifold has a finite number of truncated ``cusps" of the form $\mathbb{C}=(-\infty,x_{0}]_{x}\times \mathbb{T}^2$ with $g_{H}=dx^{2}+e^{2x}g_T$, where $g_T$ is an $x-$independent flat metric on $\mathbb{T}^2$. A cusp spacetime is a flat cone with $\mathbb{M} = \mathbb{R}_{x}\times \mathbb{T}^{2}$ and $g_{H} = dx^{2}+e^{2x}g_T$, with $g_T$ $x-$independent and flat. As mentioned a few lines above, the two ends of double cusps are asymptotic to cusp spacetimes as $t\rightarrow \infty$. This behavior is not observed in the coordinates $t,x$ but rather in new coordinates linearly related to them. On a double cusp solution, consider the new coordinates,
\begin{align}  
& t'=-\bigg(\frac{1}{2}+\frac{W_{0}^{2}}{2}\bigg)x+\bigg(\frac{3}{2}+\frac{W_{0}^{2}}{2}\bigg)t,\\
& x' = -\bigg(\frac{1}{2}+\frac{W_{0}^{2}}{2}\bigg)t + \bigg(\frac{3}{2}+\frac{W_{0}^{2}}{2}\bigg)x.
\end{align}
These coordinates are plotted in \Cref{fig:doublebehavior}. When we fix $x'$ and increase $t'$, or when we fix $t'$ and increase $x'$, both $x$ and $t$ increase. In this sense, these new coordinates are adapted to the `right' end. It is over these coordinates that the double-cusp metric approaches a cusp spacetime metric ($\tau = e^{t'}$). This is easy to show and has been done in \cite{reiris2010ground} in detail. On the `left' end, one can also define coordinates $x'',t''$ where the evolution displays a similar behavior. The whole picture is represented in \Cref{fig:doublebehavior}. This phenomenon is best observed globally along the CMC foliation. Indeed, double cusps admit a global CMC foliation of Cauchy hypersurfaces $\Sigma_{k}$ covering the whole spacetime, where the mean curvature $k$ ranges on $(-\infty,0)$, \cite{reiris2010ground}. More specifically, there is a Cauchy hypersurface $\Sigma_{-3}$ of mean curvature $-3$ defined by a graph $t=s(x)$ and any other leaf of the CMC foliation is obtained translating in $t$ the graph of $s(x)$. Furthermore, the graph of $s(x)$ approaches a $t'={\rm const}$ and $t''={\rm const}$ line as $x$ goes to $\infty$ and $-\infty$ respectively. Hence, one can simultaneously observe the convergence to the flat cones on the right and left ends by following the CMC foliation $\Sigma_{k}$, $k\uparrow 0$. If one performs CMC scalings so that the mean curvature of each leaf $\Sigma_{k}$ becomes $-3$, then a convergence-collapse picture emerges. Roughly speaking, the scaled metric over the $\Sigma_{k}$ converges to a hyperbolic cusp metric on each of the two ends (one must follow the $x'=$ const and $x''=$ const directions), while the central part collapses its volume while keeping its curvature bounded so that the narrow necks appear to look like thin and long lines. \Cref{fig:doublebehavior} depicts this phenomenon.   
\begin{figure}[h]
\centering%
    \def\svgwidth{.7\columnwidth}
    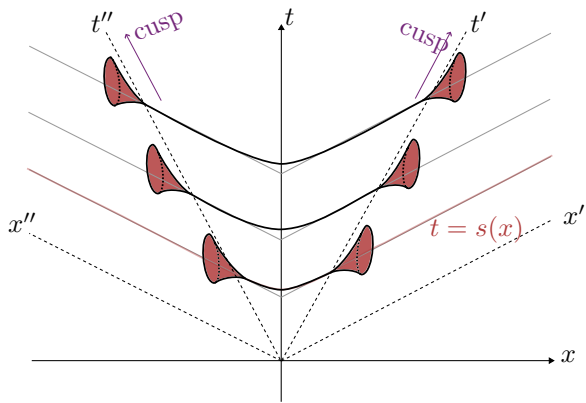

\caption{Double cusp's behavior over the CMC foliation after CMC scalings.}
\label{fig:doublebehavior}
\end{figure}
There is a relevant but standard change of variables $(t,x)\rightarrow (R,V)$ given by,
\begin{equation}
R=R_{0}e^{2t}\cosh(2x),\quad V=R_{0}e^{2t}\sinh(2x),
\end{equation}
where one can better observe certain important global facts. The spatial coordinate $V$ is the conjugate of $R$ and satisfies the wave equation too. With this change, we have,
\begin{equation}
-dt^{2}+dx^{2} = \frac{1}{4R_0^2e^{4t}}(-dR^2+dV^2) = \frac{1}{4(R^2-V^2)}(-dR^2+dV^2)
\end{equation}
so light rays in the plane $(t,x)$ are mapped into light rays in the $(R,V)$ plane. The spacetime region is the region $R>|V|$ and the slice $t=0$ maps into the hyperbola $R^{2}-V^{2}=R_{0}^{2}$, $R>0$. This is displayed in \Cref{fig:rv}. The whole picture proves that $t=0$ is a Cauchy surface for the double cusp spacetime. It can be seen that the spacetime cannot be smoothly extended to the past boundary $R=|V|$, which is singular. The CMC curve $t=s(x)$ is mapped into a curve that is also asymptotic to the lines $R=V, V>0$ and $R=-V, V<0$ but is not a hyperbola, of course. Though these two Cauchy hypersurfaces look similar, there is a clear distinction between them: the former approaches the past boundary faster than the latter.

\begin{figure}[!ht]
\centering%
    \def\svgwidth{0.6\columnwidth}
    %% Creator: Inkscape 1.2 (dc2aedaf03, 2022-05-15), www.inkscape.org
%% Accompanies image file '8.pdf' (pdf, eps, ps)
\begingroup%
  \makeatletter%
  \providecommand\color[2][]{%
    \errmessage{(Inkscape) Color is used for the text in Inkscape, but the package 'color.sty' is not loaded}%
    \renewcommand\color[2][]{}%
  }%
  \providecommand\transparent[1]{%
    \errmessage{(Inkscape) Transparency is used (non-zero) for the text in Inkscape, but the package 'transparent.sty' is not loaded}%
    \renewcommand\transparent[1]{}%
  }%
  \providecommand\rotatebox[2]{#2}%
  \newcommand*\fsize{\dimexpr\f@size pt\relax}%
  \newcommand*\lineheight[1]{\fontsize{\fsize}{#1\fsize}\selectfont}%
  \ifx\svgwidth\undefined%
    \setlength{\unitlength}{410.26919685bp}%
    \ifx\svgscale\undefined%
      \relax%
    \else%
      \setlength{\unitlength}{\unitlength * \real{\svgscale}}%
    \fi%
  \else%
    \setlength{\unitlength}{\svgwidth}%
  \fi%
  \global\let\svgwidth\undefined%
  \global\let\svgscale\undefined%
  \makeatother%
  \begin{picture}(1,0.51767105)%
    \lineheight{1}%
    \setlength\tabcolsep{0pt}%
    \put(0,0){\includegraphics[width=\unitlength,page=1]{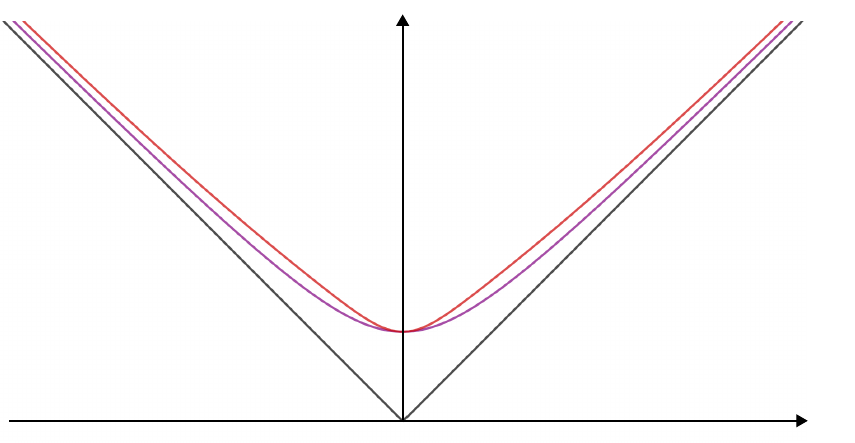}}%
    \put(0.49221623,0.47937267){\color[rgb]{0,0,0}\makebox(0,0)[lt]{\lineheight{1.25}\smash{\begin{tabular}[t]{l}$R$\end{tabular}}}}%
    \put(0.90740701,0.0478786){\color[rgb]{0,0,0}\makebox(0,0)[lt]{\lineheight{1.25}\smash{\begin{tabular}[t]{l}$V$\end{tabular}}}}%
    \put(0.47754312,0.25656086){\color[rgb]{0.85882353,0.30196078,0.30196078}\makebox(0,0)[lt]{\lineheight{1.25}\smash{\begin{tabular}[t]{l}$CMC$\end{tabular}}}}%
    \put(0.60742614,0.12993856){\color[rgb]{0.65882353,0.31372549,0.65882353}\makebox(0,0)[lt]{\lineheight{1.25}\smash{\begin{tabular}[t]{l}$t=const$\end{tabular}}}}%
    \put(0,0){\includegraphics[width=\unitlength,page=2]{8.pdf}}%
  \end{picture}%
\endgroup%

\caption{Different Cauchy surfaces seen in the $R-V$ coordinates.}
\label{fig:rv}
\end{figure}

Based on the discussion above, it may seem that making the analysis on the coordinates $(t,x)$ is inconvenient. However, the great advantage of it is that it displays $(W_{b},q_{b})$ as a time-independent non-affinely parametrized geodesic segment in the hyperbolic plane, which proves to be highly useful. Given the motivation of the stability problem we have discussed, the natural gauge to work would be the CMC gauge. This approach, however, introduces several complexities, for instance, finding and dealing with the right shift, which complicates the analysis. We plan to address this in future work.      

The rest of the paper is organized as follows. In  \cref{SMR}, we state the main results about the evolution of $R$, $q$ and $W$ to be proved. We first state the future long time existence theorem. This theorem will be proved in section \ref{sec:future_existence}. Then, in the same section \ref{SMR}, we will state a theorem summarizing what we will prove (and need) regarding the evolution of $R$. Finally we state a stability theorem for $W$ assuming $q=0$, namely polarized evolution, and then a stability theorem for non-polarized evolution.  Also in \cref{SMR}, and assuming the validity of the statements, we discuss the evolution of $a$ and the stability of the double cups as spacetimes. As $R$ satisfies the wave equation, its analysis is relatively standard. The study of the evolution of $q$ and $W$ requires however much more analysis. Section \ref{sec:pol} is dedicated to prove the stability result for smooth polarized initial data differing from the background data on a compact set, and section \ref{sec:nonpol} is dedicated to prove the stability statement for smooth non-polarized initial data differing from the background data on a compact set. The main statements for general initial data is done basically through a standard completion argument and is carried out in section \ref{sec:gen}.

\section{Statements of the main results}\label{SMR}
%%%%%%%%%%%%%%%%%%%

We introduce first a few norms and spaces that will be used to control the evolution and to formulate the main results. After that we state first the long-time existence theorem and the stability theorems for $R$ and $(W,q)$. Finally, assuming the validity of these statements, we discuss the evolution of $a$ and the stability of the double cusps as spacetimes. 

\subsection{Norms and spaces used in the article.}
%%%%%%%%%%%%%%%%%%%%%%%%%

We define,
\be
m_0(t) := \norm{R-R_b}_{C^0}+ \norm{R_t - {R_b}_t}_{C^0},
\ee
and for all $k\geq 1$ we define,
\be
m_k(t) := \norm{R-R_b}_{C^k} + \norm{R_t - {R_b}_t}_{C^{k-1}},
\ee
where $R, R_{b}, R_{t}$ and $R_{bt}$ are considered at time $t$. The quantity $m_{0}(t)$ measures the $C^{0}$ norm between $R$ and $R_{b}$ and between $R_{t}$ and $R_{bt}$. The quantity $m_{k}(t)$ instead measures, at time $t$, the $C^{k}$ norm between $R$ and $R_{b}$ and the $C^{k-1}$ norm between $R_{t}$ and $R_{bt}$. We also define for $k\geq 1$,
\begin{gather}
\widetilde{\mathcal{M}}_k(t) := \norm{W-W_b}_{\widetilde{H}_k} + \norm{\partial_t(W-W_b)}_{\widetilde{H}_{k-1}} + \norm{q}_{\widetilde{H}_k} + \norm{\partial_t q}_{\widetilde{H}_{k-1}},
\end{gather}
where again, the functions inside the norms are considered at time $t$. Here $\widetilde{H}_k$ is the completion of the space of smooth and compactly supported functions, $C^\infty_c(\R)$, with respect to the norm, 
\begin{equation}
\norm{f}_{\widetilde{H}_k}^2 := \sum_{i=0}^{k}\int_\R (f^{(k)}(x))^2 \cosh(2x) \ dx.
\end{equation} 
This is a weighted Sobolev space with weight $\cosh(2x)$. Lastly, let $C_0^k(\R)$ be the space of $C^k$ functions $f$ such that for every $i\leq k$, $f^{(i)}(x) \to 0$ when $\abs{x}\to +\infty$. 

Throughout the work, by initial data we refer to the functions $R(0,\cdot)$, $R_t(0,\cdot)$, $W(0,\cdot)$, $W_t(0,\cdot)$, $q(0,\cdot)$, $q_t(0,\cdot)$.

\subsection{Statements on the long-time existence and the stability of $R$ and $(W,q)$.} 
%%%%%%%%%%%%%%%%%%%%%%%%%%%%%%%%%%%%%%%%%%%%

The long-time existence statement that we will prove is the following,
\begin{theorem}[Future long-time existence]\label{th:existencia_todo_tiempo} \mbox{}
Consider initial data to the \cref{eq:sysR,eq:sysW,eq:sysq}, such that $R(0,\cdot)$, $W(0,\cdot)$, $q(0,\cdot)$ $\in C^2$ and $R_t(0,\cdot),W_t(0,\cdot),q_t(0,\cdot)$ $\in C^1$. Suppose additionally that $m_0(0)<2R_0/3$. Then there is a unique $C^2$ solution defined for all $t \geq 0$.
\end{theorem}
Observe that besides the condition on $R$, no other condition is imposed on the data. This theorem is proved independently from the stability results stated below. It is proved in section \ref{sec:future_existence}.

The next statement summarizes the main results that we will prove about the evolution of $R$. The proof of this theorem follows directly from \cref{prop:R1,lemma:coeff_estimates,lemma:coef_est_2}.

Let $\alpha =(m,n)$ denote a multi-index, and let $\abs{\alpha} := m+n$. Recall that, 
\begin{equation}
R_b = R_0 e^{2t}\cosh(2x).
\end{equation}
\begin{theorem}\label{thm:RR}
For every $C^2$ solution $R$ to the \cref{eq:sysR}, if $m_0(0)<2R_0/3$, then $R(t,x)>0$ when $t\geq 0$. Moreover,
\begin{equation}\norm{R-R_b}_\infty (t) \leq (t+1)m_0(0) \ \ \text{for } t \geq 0,
\end{equation}
and for all multi-index $\alpha \neq (0,0)$, there is a constant $C$ such that, for $t\geq 0$, 
\begin{equation}
\norm{\partial^\alpha (R-R_b)}_{\infty}(t) \leq m_{\abs{\alpha}}(0),
\end{equation} 
for $C^{\abs{\alpha}}$ solutions, and
\begin{equation}
\abs{\partial^\alpha \left(\frac{R_t}{R}-\frac{{R_b}_t}{{R_b}}\right)}(t,x),\abs{\partial^\alpha \left(\frac{R_x}{R}-\frac{{R_b}_x}{R_b}\right)}(t,x) \leq C \frac{t+1}{e^{2t}\cosh(2x)}m_{\abs{\alpha}+1}(0),
\end{equation} 
for $C^{\abs{\alpha} +1}$ solutions.
\end{theorem}
The following statement is the main result we will prove regarding the evolution of $W$ for polarized perturbations, i.e., with $q=0$. It demonstrates, in particular, that within the class of polarized solutions, the polarized double cusp acts as a global attractor and that every perturbation converges to the double cusp exponentially fast. The proof of this theorem proceeds by first establishing the result for smooth initial data that differ from that of the background only on a compact set, followed by a completion argument. \Cref{sec:pol} focuses on proving the result for smooth, compactly supported perturbations, and is achieved by \Cref{cor:polarized_compact_sup}. The completion argument is carried out in \Cref{sec:gen}.

%The essential part of the proof consists in proving the result for smooth compactly supported perturbations, which is achieved by \Cref{cor:polarized_compact_sup} in \cref{sec:pol}.  % \textcolor{red}[A version of this theorem for compacty supported perturbations is proved in \Cref{cor:polarized_compact_sup} section \ref{sec:pol}. For arbitrary data it is proved in section \ref{sec:gen}.}

\begin{theorem}[Stability for polarized perturbations]\label{th:polarized_perturbations}
Let $k\geq 3$ and assume $m_0(0)<2R_0/3$. Then, for any initial data to the \cref{eq:sysR,eq:sysW,eq:sysq}, satisfying $(R-R_b)(0,\cdot) \in C_0^k$, $\partial_t(R-R_b)(0,\cdot) \in C_0^{k-1}$, $(W-W_b)(0,\cdot) \in \widetilde{H}^k$, $\partial_t(W-W_b)(0,\cdot) \in \widetilde{H}^{k-1}$, and $q(0,\cdot) = q_t(0,\cdot) = 0$, there is a unique $C^{k-1}$ solution to the \cref{eq:sysR,eq:sysW,eq:sysq}, defined for all time $t \geq 0$, with $q$ identically zero, and
\begin{equation}
\widetilde{\mathcal{M}}_k(t) \leq Ce^{-t}(t+1)\left(\widetilde{\mathcal{M}}_k(0) + m_k(0)\right).\label{eq:controlpW}
\end{equation}
Moreover, the constant $C$ depends only on an upper bound on $m_k(0)$ and $k$.
\end{theorem}

In this statement, the phrase \textit{Moreover, the constant $C$ depends only on an upper bound on $m_k(0)$ and $k$}, means that, for any solution having $m_k(0)<B$, for some constant $B$, we can use the same constant $C$. This type of wording will be particularly useful when dealing with a sequence of solutions in \cref{sec:gen}. We will use similar expressions throughout the work.

For non-polarized perturbations ($q\neq 0$) a similar estimate is obtained but only for $\widetilde{\mathcal{M}}_{3}(t)$. However, this time we require the perturbation to be small enough at $t=0$. The following is our main result. As above, we first establish the result for data that differ from that of the background only on a compact set. This is done in \Cref{cor:nonpol_comp_2} of \cref{sec:nonpol}. The completion argument is carried out in \cref{sec:gen}.

%A version of this theorem for compacty supported perturbations is proved in section \ref{sec:nonpol}. For arbitrary data is proved in \cref{sec:gen}.

\begin{theorem}[Stability for non-polarized perturbations]\label{th:nonpol_perturbations}
Assume $m_0(0) < 2R_0/3$. Then, for any initial data to the \cref{eq:sysR,eq:sysW,eq:sysq}, satisfying $(R-R_b)(0,\cdot) \in C_0^3$, $\partial_t(R-R_b)(0,\cdot) \in C_0^{2}$, $(W-W_b)(0,\cdot) \in \widetilde{H}^3$, $\partial_t(W-W_b)(0,\cdot) \in \widetilde{H}^{2}$, $q(0,\cdot) \in \widetilde{H}^3$, and $q_t(0,\cdot)$ in $\widetilde{H}^2$, there is a unique $C^2$ solution defined for all time $t \geq 0$. Furthermore, there exists $\delta>0$, independent of the solution, such that if $m_3(0)<\delta$, $\widetilde{\mathcal{M}}_3(0)<\delta$, then
\begin{equation}
\widetilde{\mathcal{M}}_3(t)\leq Ce^{-t}(t+1)(\widetilde{\mathcal{M}}_3(0) + m_3(0)). \label{eq:controlWq}
\end{equation}
Moreover, the constant $C$ depends only on an upper bound on $\delta$.
\end{theorem}
Even though the stability statements are written for the polarized double cusp, by the observation made in \Cref{fig:asymptotic_conv}, it is easy to see that they hold for every double cusps.

\subsection{Statements on the evolution of $a$ and the stability of the double cusps.} 
%%%%%%%%%%%%%%%%%%%%%%%%%%%%%%%%%%%%%%%%%%

A basic consequence of the results presented in the previous section is that if the initial data for $(R,W,q)$ satisfies the hypotheses of \Cref{th:existencia_todo_tiempo}, then $a$ is defined for all $t\geq 0$. This happens because the \cref{eq:wave_a} is a wave equation with a source defined for all $t\geq 0$. Moreover, if the constraint equations are satisfied at $t=0$, then, they are satisfied for every $t\geq 0$. This directly implies the following corollary. 

\begin{corollary}\label{cor:11}
Suppose that $R,R_t,W,W_t,q$ and $q_t$, at $t=0$, satisfy the hypotheses of \Cref{th:existencia_todo_tiempo}, and also that $a(0,\cdot) \in C^2$ and $a_t(0,\cdot) \in C^1$. Assume that the \cref{eq:a_fuerte,eq:a_debil} are satisfied at $t=0$. Then, the solution to the \cref{eq:sysR,eq:sysW,eq:sysq,eq:wave_a} is defined for all $t\geq 0$. Furthermore, the metric, \cref{eq:metricform}, on $[0,+\infty)\times \R \times \T^2$, gives a Cauchy development of the data.
\end{corollary}

Regarding the stability of $a$, note that the \cref{eq:wave_a} implies
\begin{equation}
(a-a_{b})_{tt}-(a-a_{b})_{xx}= G-G_b + F(W,W_t,W_x,q_x,q_t)-F(W_{b},{W_{b}}_t,{W_b}_x, {q_{b}}_x,{q_b}_t),
\end{equation}
where $G = (R_t^2-R_x^2)/4R^2$. It can be seen that, due to \Cref{thm:RR}, $G-G_b$ is controlled by $\frac{C (t+1)}{e^{2t}\cosh(2x)}m_1(0)$ (see \cref{eq:G_est}, \Cref{prop:R1}). The other part of the source is easily  controlled by the \cref{eq:controlWq}. Using this and D'Alembert, we deduce that the contribution of the source to $\|a(t,\cdot)-a_{b}(t,\cdot)\|_{C^{0}}$ is controlled by $\widetilde{\mathcal{M}}_{2}(0)+m_{2}(0)$. The contribution of the homogeneous solution is controlled by the initial data norm $\|a(0,\cdot)-a_{b}(0,\cdot)\|_{C^{0}}+\|a_{t}(0,\cdot)-{a_t}_{b}(0,\cdot)\|_{L^{1}}$.  This immediately leads to the following results.

\begin{theorem}\label{thm:astab}
In the hypotheses of \Cref{th:nonpol_perturbations}, for each $\epsilon>0$ we can find $\delta>0$ such that
\begin{enumerate}
\item If the initial data for $(R,W,q,a)$ satisfies $\widetilde{\mathcal{M}}_{3}(0)+m_{3}(0)\leq \delta$ and $\|a(0,\cdot)-a_{b}(0,\cdot)\|_{C^{0}}+\|a_{t}(0,\cdot)-{a_{b}}_t(0,\cdot)\|_{L^{1}}\leq\delta$, then $\|a(t,\cdot)-a_{b}(t,\cdot)\|_{C^{0}}\leq \epsilon \ \forall t \geq 0$.
\item If the initial data for $(R,W,q,a)$ satisfies $\widetilde{\mathcal{M}}_{3}(0)+m_{3}(0)\leq \delta$ and $\|a_x(0,\cdot)-{a_{b}}_x(0,\cdot)\|_{C^{0}}+\|a_t(0,\cdot)-{a_{b}}_t(0,\cdot)\|_{C^0}\leq\delta$, then $\|a_x(t,\cdot)-{a_{b}}_x(t,\cdot)\|_{C^{0}} + \|a_t(t,\cdot)-{a_{b}}_t(t,\cdot)\|_{C^{0}} \leq \epsilon \ \forall t \geq 0$.
\end{enumerate}
\end{theorem}

It can be easily seen that the spacetime perturbations satisfying the second point are future geodesically complete, provided that $\delta$ is small enough. The natural question is whether there is a non-trivial perturbation on these spaces subjected to the constraint equations. This indeed is the case, for instance, if we require stricter norms for $(R,W,q)$ at $t=0$. By doing this, $(a - a_b)(0, x)$ and $(a - a_b)_t(0, x)$ can be solved for in the constraint \cref{eq:a_fuerte,eq:a_debil}, and naturally belong to the spaces above initially, and then for each \( t \geq 0 \). An example of such norms could be $\norm{f}_{\widetilde{H}_{p,k}}^2 := \sum_{i=0}^{k}\int_\R (f^{(k)}(x))^2 \cosh^{p}(2x) \ dx$ for $W$ and $q$, and ${m_{l,k}}(f) := m_{k}(\cosh^{l}(2x)f(x))$ for $R$.
With these norms, using $p=2$ instead of $p=1$, and $l=1$ instead of $l=0$, one can see that small perturbations in this new sense imply that $a_t(0,x)$ and $a_x(0,x)$ can be solved out in the \cref{eq:a_fuerte,eq:a_debil}. Hence, by \Cref{thm:astab}, $\|a(t,\cdot)-a_{b}(t,\cdot)\|_{C^{1}}$ is finite for each $t\geq 0$, and arbitrarily small by reducing the values of $m_{l,1}(R-R_b)(0)$ and $\widetilde{\mathcal{M}}_{2,3}(0)$, where
\begin{equation}
\widetilde{\mathcal{M}}_{p,k}(t) := \norm{W-W_b}_{\widetilde{H}_{p,k}} + \norm{\partial_t(W-W_b)}_{\widetilde{H}_{p,k-1}} + \norm{q}_{\widetilde{H}_{p,k}} + \norm{\partial_t q}_{\widetilde{H}_{p,k-1}}.
\end{equation}

%The proofs of the \Cref{th:polarized_perturbations,th:nonpol_perturbations}, are done first for \textit{smooth compactly supported perturbations}, i.e., solutions with initial data differing only on a compact set from that of the background. This is done in \cref{sec:pol,sec:nonpol}. The latter section is the central part of the paper. Finally, in section \ref{sec:gen}, we give a rather general argument to extend this simplified versions to larger functional spaces, proving \Cref{th:polarized_perturbations,th:nonpol_perturbations}. Regarding the future long-time existence, \Cref{th:existencia_todo_tiempo}, this is proven in section \ref{sec:future_existence}.

\section{Proof of the stability of compactly supported polarized perturbations}\label{sec:pol}
%%%%%%%%%%%%%%%%%%%%%%%%%%%%%%%%%%%%%%%%%%%%%

In this section, we shall address smooth compactly supported perturbations with $q=0$, namely, a smooth solution $(R, W)$ to the \cref{eq:sysR,eq:sysW}, with $q=0$, whose initial data differs from that of the background only on a compact set. The goal is proving \Cref{cor:polarized_compact_sup}. Many of the computations in this section will be used again for the non-polarized case. 

Throughout the work, depending on the kind of computations, we will use $\partial_x f$ or $f_x$. Also $\partial^{\alpha}f$ with $\alpha$ a multi-index.  
We say that a function $f(t,x)$ is  \textit{of locally x-compact support} if for any interval $[T_1,T_2]$ there is a compact subset $K\subset \R$ such that $f(t,x) = 0$ if $t\in [T_1,T_2]$ and $x \notin K$. Recall the notation $m_k$. The following basic proposition will be used later.
\begin{proposition}\label{prop:R1}
If $R$ is a $C^2$ solution such that $m_0(0)<\frac{2}{3}R_0$, then the solution is defined for every $t\geq 0$, and $R>0$ for $t\geq 0$. Furthermore, if $(R,W)$ is a smooth solution whose initial data differs only on a compact set from that of the background solution, then $(R-R_b, W-W_b)$ is of locally x-compact support.
\end{proposition}
\begin{proof}
$R$ is clearly defined for every $t$. The use of D'Alembert's formula gives $R \geq R_b - m_0(0)(t+1), t \geq 0$. With this, it can be seen that both claims about $R$ are true. Fixed $R$, the equation for $W$ is linear, and then is defined for all time $t \geq 0$. For the last assertion about $W$, use a finite speed propagation argument.
\end{proof}

\subsection{A change of variables and the basic energy estimate.}
%%%%%%%%%%%%%%%%%%%%%%%%%%%%%%%%

To prove the asymptotic stability, we need some useful energies. Consider the change of variable $z = R^{1/2}(W-W_b)$. In this new variable, the equation for $W$ becomes
\begin{equation}
z_{tt} - z_{xx} + z G = g  \ \ \ \textrm{with } G = \frac{R_t^2 - R_x^2}{4R^2} \textrm{ and } g = R^{1/2}\left(\frac{R_x}{R}-\frac{{R_b}_x}{R_b}\right){W_b}_x.
\label{eq:z}
\end{equation}
Now the background corresponds to $z_b=0$, and $z$ is of locally x-compact support if and only if $W-W_b$ is of locally $x-$compact support. In this case, the energy 
$$E(t):=\frac{1}{2}\int_\R z_t^2 + z_x^2 + z^2G_b \ dx,$$ is well-defined, where $G_b = \tfrac{1}{\cosh^2(2x)}$. Since $G$ is well approximated by $G_b$ (as we shown below), this energy is natural, as is the Lagrangian energy when $R = R_b$. 

We need some estimates involving $R-R_b$ so that we can control $E$. These properties are given in the following lemma.
\begin{proposition}[Coefficients estimates]\label{lemma:coeff_estimates}
Consider $R$, a $C^2$ solution. For all $t\geq 0$, we have the following estimates:
\begin{gather}
\norm{R-R_b}_\infty (t) \leq (t+1)m_0(0), \label{eq:R_est} \\
\norm{R_x-{R_b}_x}_\infty (t), \norm{R_t - {R_b}_t}_\infty(t) \leq  m_1(0). \label{eq:RxRt_est}
\end{gather}
Furthermore, as we are assuming $m_0(0) \leq \frac{2}{3}R_0$ we also have that $R \sim R_b$, i.e, there is a constant $d>0$ such that
\begin{equation}
\frac{1}{d} \leq \norm[\bigg]{\frac{{R_b}}{R}}_\infty(t) \leq d,
\label{eq:R_equiv}
\end{equation}
and this, in turn, implies the existence of a numeric constant $C>1$ such that
\begin{gather}
\abs{\frac{R_t}{R}-\frac{{R_b}_t}{R_b}} (t,x) , \abs{\frac{R_x}{R}-\frac{{R_b}_x}{R_b}}(t,x) \leq \frac{C(t+1)}{e^{2t}\cosh(2x)}m_1(0), \label{eq:RtbyR_RxbyR_est}\\
\abs{G_b - G}(t,x) \leq \frac{C (t+1)}{e^{2t}\cosh(2x)}m_1(0), \label{eq:G_est}\\
\abs{g}(t,x) \leq \frac{C (t+1)}{e^t \cosh^{3/2}(2x)}m_1(0). \label{eq:g_est}
\end{gather}
\end{proposition}
\begin{proof}
Use D'Alembert to differentiate the first estimates. The other follows from direct computation using these.
\end{proof}
\begin{lemma}[The basic energy inequality]\label{prop:basic_energy_inequality}
Let $(R,z)$ be a smooth solution to the \cref{eq:sysR,eq:z}, whose initial data differs only on a compact set from that of the background solution, and such that $m_0(0)<2R_0/3$. Then there is a constant $C$ such that, 
\begin{equation}
E^{1/2}(t) \leq C(E^{1/2}(0) + m_1(0)), \ \text{with }\ t\geq 0.
\label{eq:basic_energy_inequality}
\end{equation}
Furthermore, the constant $C$ depends only on a bound on $m_1(0)$.
\end{lemma}
\bibliographystyle{plain}
\begin{proof}
The function $z$ is of locally x-compact support as $W-W_b$ is. For this reason, we can differentiate under the integral and use integration by parts. Then, we use the equation to obtain

$$
\begin{aligned}
\dot{E} &= \int_\R z_t z_{tt} + z_x z_{xt} + zz_tG_b \ dx = \int_\R z_t(z_{tt}-z_{xx} + zG_b) \ dx  \\
& = \int_\R z_t z(G_b-G)\ dx + \int_\R z_t g \ dx \\ 
&\leq \underbrace{\sqrt{\int_\R z_t^2 \ dx}\sqrt{\int_\R z^2(G_b-G)^2 \ dx}}_{\text{First term}} + \underbrace{\sqrt{\int_\R z_t^2 \ dx}\sqrt{\int_\R g^2 \ dx}}_{\text{Second term}}. \\
\end{aligned}
$$
In the first term, the first integral is bounded by $\sqrt{2}E^{1/2}$. For the second integral, use the estimate (\ref{eq:G_est}) to find out that this term is less or equal to
$$C(t+1)m_1(0) e^{-2t}E^{1/2} \sqrt{\int_\R \frac{z^2}{\cosh^2(2x)} \ dx}, $$ 
but the integrand in the last integral is just $4z^2G_b$, so in the end, our first term is less or equal to
$$C(t+1)e^{-2t}m_1(0)E,$$ 
where we have adjusted the constant $C$. Let us control the second term. This term is a product of two integrals. The first one is less or equal to $E^{1/2}$. For the second integral, the use of the estimate (\ref{eq:g_est}) yields
$$\int_\R g^2 \ dx \leq C(t+1)e^{-t}m_1(0). $$ 
Using all these observations, we get
\begin{equation}
\begin{aligned}
\dot{E} &\leq C(t+1)e^{-2t}m_1(0)E + C(t+1)e^{-t}m_1(0) E^{1/2} \\
& \leq C(t+1)e^{-2t}m_1(0)E + C(t+1)e^{-t}E  + C(t+1)e^{-t}m_1^2(0) \\
& \leq C(t+1)e^{-t}E + C(t+1)e^{-t}m_1^2(0),
\end{aligned} 
\label{eq:casi_basic_energy_inequality}
\end{equation}
where in the last inequality, the constant depends on $m_1(0)$. This inequality implies the thesis.
%Now we would like to divide by $E^{1/2}$ and apply the Gronwall-Belman inequality. However, at first, $E$ could be zero at some $t$. To avoid this problem, define, for $\epsilon >0$, $E_\epsilon:= E + \epsilon$. Note that we still have  \ref{eq:casi_basic_energy_inequality}
%and now we are allowed to divide by $E_\epsilon^{1/2}$, obtaining
%$$\dot{E_\epsilon^{1/2}} \leq C(t+1)e^{-2t}m_1(0)E_\epsilon^{1/2} + C(t+1)e^{-t}m_1(0)$$
%Integrating, using that $(s+1)e^{-s} \in L^1(s)$ and applying Gronwall we find
%$$
%\begin{aligned}
%E_\epsilon^{1/2}(t) &\leq (E_\epsilon^{1/2}(0) + Cm_1(0))\text{exp}\left(\int_0^t C(s+1)e^{-2s}m_1(0) \ ds \right) \\
%& \leq (E_\epsilon(0)^{1/2} + Cm_1(0))\text{exp}\left( Cm_1(0) \right) \\
%\end{aligned} 
%$$
%Until now, the constant we have been adjusting from line to line does not depend on anything. However, in the next step, we bound the exponential and called the new bound $C$ again. To do this, we need to bound the argument of the exponential, and for this reason, we get the dependence of $C$ on $m_1(0)$. We obtain
%$$E_\epsilon^{1/2}(t) \leq C(E_\epsilon(0)^{1/2} + m_1(0))$$
%letting $\epsilon \to 0$, we get the first inequality stated in the theorem. 
\end{proof}
Note that since $R \sim R_b = R_0e^{2t}\cosh(2x)$, this implies exponential decay in our original variables plus decay as $x$ goes to infinity. 

\subsection{Higher order energy estimates.} Let $\alpha$ be a multi-index $\alpha = (m,n)$. The first letter, $m$, will refer to time derivatives wereas $n$ will refer to spatial derivatives.
Let us denote $$E^\alpha(t):= \frac{1}{2}\int_\R (\partial^\alpha z)_t^2 + (\partial^\alpha z)_x^2 + (\partial^\alpha z)^2G_b \ dx.$$ 
The most important of these energies are the ones with $\alpha = (0,m)$ because they are related to $H^k$ norms of $z$ and $\partial_t z$. This fact is important since, at $t=0$, they involve only the initial data. One could be tempted to do an argument similar to the one made in the derivation of the \cref{eq:basic_energy_inequality}, but using $E^{(0,m)}$. If we do this, we will discover that the growth of $E^{(0,m)}$ is bounded by a polynomial of degree $m$. This result is not bad, but we found another, longer way to obtain better estimates. First, we derive estimates for $E^{(n,0)}$, and then we pass our estimates to $E^{(0,m)}$ using the equations satisfied by $z$ and its derivatives. Now, more coefficient estimates are needed.  

\begin{proposition}\label{lemma:coef_est_2}
Suppose $m_0(0)<2R_0/3$. \vspace{0.2 cm} \newline 
Assuming $R$ is a $C^{\abs{\alpha}}$ solution:
\begin{enumerate}[a), wide, labelindent=0 cm]
\item For all multi-index $\alpha \neq (0,0)$ we have
\begin{equation}
\norm{\partial^\alpha (R-R_b)}_{\infty}(t) \leq m_{\alpha}(0) \ \ t \geq 0.
\label{eq:Ralpha_est}
\end{equation}
\item For every multi-index $\alpha$, $\partial^{\alpha}R/R$ and $\partial^{\alpha}R^{1/2}/R^{1/2}$ are bounded to the future. Moreover, the bound depends only on a bound on $m_{\abs{\alpha}}(0)$. \vspace{0.1 cm} \newline
Assuming $R$ is a $C^{\abs{\alpha}+1}$ solution:
\item For every multi-index $\alpha$, $\partial^{\alpha}R_t/R$, $\partial^{\alpha}R_x/R$, $\partial^\alpha (R_t/R)$ and $\partial^{\alpha}(R_x/R)$ are bounded to the future. Moreover, the bound depends only on a bound on $m_{\abs{\alpha}+1}(0)$. \vspace{0.1 cm}
\item For all multi-index $\alpha$ there is a constant $C$ such that, $\forall t\geq 0$ 
\begin{equation} \abs{\partial^\alpha \left(\frac{R_t}{R}-\frac{{R_b}_t}{{R_b}}\right)}(t,x),\abs{\partial^\alpha \left(\frac{R_x}{R}-\frac{{R_b}_x}{R_b}\right)}(t,x) \leq C \frac{t+1}{e^{2t}\cosh(2x)}m_{\abs{\alpha}+1}(0).
\label{eq:RalphabyR_est}
\end{equation}
\item Estimates for $G_b-G$: for every multi-index $\alpha$ there is a constant $C$ such that
\begin{equation}
\abs{\partial^{\alpha}(G-G_b)}(t,x) \leq C \frac{(t+1)}{e^{2t}\cosh(2x)}m_{\abs{\alpha}+1}(0) \ \ \forall t \geq 0.
\label{eq:Galpha_est}
\end{equation}
Moreover, the constant depends only on a bound on $m_{\abs{\alpha}+1}(0)$.
\item For every multi-index $\alpha$ there is a constant $C$ such that
\begin{equation}
\abs{\partial^{\alpha} G}(t,x)\leq \frac{C}{\cosh{2x}} = 2C\sqrt{G_b} \ \ \forall t \geq 0.
\label{eq:G_derivative_est}
\end{equation}
Furthermore, the constant $C$ depends on $\alpha$ and on a bound on $m_{\abs{\alpha}+1}(0)$.
\item For every multi-index $\alpha$ there is a constant $C$ such that $\abs{\partial^\alpha {W_b}_x}\leq C/\cosh(2x)$. In addition, if $\alpha$ is not purely spatial then $\partial^{\alpha} {W_b}_x = 0$.
\item Estimates for $g$: For every multi-index $\alpha$ there is a constant $C$ such that
\begin{equation}
\abs{\partial^\alpha g}(t,x) \leq C\frac{t+1}{e^{2t}\cosh^{3/2}(2x)}m_{\abs{\alpha}+1}(0) \ \ \forall t \geq 0.
\label{eq:galpha_est} 
\end{equation}
Additionally, the constant $C$ just depends on a bound on $m_{\abs{\alpha}+1}(0)$. 
\end{enumerate}
\end{proposition}
\begin{proof}
Item a) is a direct consequence of D'Alembert's Formula. Item b) and c) are just computations using a) and the fact that $R \sim R_b$. For these computations, it is often helpful to use recursion formulas, such as
$$
\partial^{\alpha}\left(\frac{R_t}{R}\right) = \frac{\partial^\alpha R_t}{R} - \sum_{0\leq \beta<\alpha} \binom{\alpha}{\beta}\frac{\partial^{\alpha-\beta}R \partial^{\beta}(\tfrac{R_t}{R})}{R},
$$
or
$$\abs{\frac{\partial^\alpha(R^{1/2})}{R^{1/2}}}\leq \abs{\frac{\partial^\alpha R}{2 R}} + \sum_{0<\alpha<\beta}\binom{\alpha}{\beta} \abs{\frac{\partial^{\alpha-\beta}R^{1/2}\partial^{\beta}(R^{1/2})}{2R}}.$$
Item d) is proved similarly, and item e) is a consequence of a)-d). Item f) follows from item e). Item g) follows from direct computation, and finally, g) is a consequence of d), g), and b).
\end{proof}
\begin{lemma}[Energy estimates for time derivatives]\label{prop:Et_estimate}
Let $(R,z)$ be a smooth solution whose initial data differs only on a compact set from that of the background solution, and such that $m_0(0)<2R_0/3$, and let $\alpha = (m,0)$. Then, there is a constant $C$ such that,
$$\sqrt{E^{(m,0)}}(t) \leq C(\sqrt{E^{(m,0)}}(0) + \ldots + \sqrt{E^{(1,0)}}(0) + \sqrt{E}(0) + m_{\abs{\alpha}+1}(0)), \text{ for } t\geq 0.$$
Furthermore, the constant $C$ just depends on $m$ and on a bound on $m_{\abs{\alpha}+1}(0)$. 
\end{lemma}
\begin{proof}
We have already proved the case $\alpha = (0,0)$ in the \Cref{prop:basic_energy_inequality}. Let us proceed by induction. Deriving the equation m-times with respect to $t$ yields:
$$
(\partial^{m}_t z)_{tt} - (\partial^{m}_t z)_{xx} + (\partial^m_t z) G + \sum_{i=0}^{m-1} \binom{\alpha}{\beta} \underbrace{\partial_t^i z}_{\substack{\text{inside } E^{(i,0)} \\ \text{already controlled}}} \underbrace{\partial_t^{m-i}(G-G_b)}_{\text{use \cref{eq:Galpha_est}}} = \underbrace{\partial^m_t g}_{\text{\text{use \cref{eq:galpha_est}}}}.$$
Here, we have used that $\partial_t^i G_b = 0$ in the last term before the equal sign.
Now, as we did in the proof of the
\Cref{prop:basic_energy_inequality}, we can differentiate $E^{(m,0)}$ with respect to time, integrate by parts, use the equation and control each of the terms that appear. The new terms are the ones that have a curly bracket in the equation above. Below these brackets, it is specified how to control these terms.  
\end{proof}
The following lemma goes in the direction of proving the desired estimates for $E^{(0,n)}$. Here we use the notation $\abs{\alpha} = m+n$ for $\alpha = (m,n)$.
\begin{proposition}\label{lemma:key_formula}
In the assumptions of \cref{prop:Et_estimate}, suppose $n\geq 1$ and $m\geq 0$, then there is a constant $C$ such that,
$$\sqrt{E^{(m,n)}}(t) \leq C\sqrt{E^{(m+1,n-1)}}
(t)+ C\sum_{0 \leq \beta \leq (m,n-1)} \sqrt{E^{\beta}}(t) + Cm_{\abs{\alpha}}(0) \ \text{for } t\geq 0,$$
where $\alpha = (m,n)$.
Now if $m\geq 1$ and $n\geq 0$ then there is a constant $C$ such that, 
$$\sqrt{E^{(m,n)}}(t) \leq C\sqrt{E^{(m-1,n+1)}}
(t)+ C\sum_{0 \leq \beta \leq (m-1,n)} \sqrt{E^{\beta}}(t) + Cm_{\abs{\alpha}}(0) \ \text{for } t\geq 0.$$

The constants just depend on a bound on $m_{\abs{\alpha}+1}(0)$, and on $m$ and $n$. 
\end{proposition}
\begin{proof}
$$
\begin{aligned}
E^{(m,n)}(t) &= \frac{1}{2}\int_\R \underbrace{[(\partial_t^m\partial_x^n z)_t]^2}_{\text{It is inside } E^{(m+1,n-1)}}+ [(\partial_t^m\partial_x^n z)_x]^2+ \underbrace{(\partial_t^m\partial_x^n z)^2}_{\text{It is inside } E^{(m,n-1)}}\underbrace{G_b}_{\leq 1} \ dx \\
& \leq E^{(m+1,n-1)}(t) + E^{(m,n-1)}(t) + \frac{1}{2}\int_\R [(\partial_t^m\partial_n^x z)_x]^2 \ dx .\\
\end{aligned}$$
To bound the last term, differentiate the \cref{eq:z} $m-$times with respect to time and $n-1-$times with respect to $x$. The differentiated equation gives us the following estimate, 
$$\int_\R (\partial^m_t \partial_x^{n+1}z)^2  \leq C \underbrace{\int_\R(\partial_t^{m+2}\partial_x^{n-1} z)^2}_{\text{It is in } E^{(m+1,n-1)}} + C \sum_{\scriptscriptstyle \beta = (0,0)}^{\scriptscriptstyle (m,n-1)} \underbrace{\int_\R (\partial^\beta z \partial^{\alpha-\beta}G)^2}_{\leq C E^{\beta}} + \underbrace{\int_\R(\partial_t^m \partial_x^{n-1}g)^2}_{\leq C m_{\abs{\alpha}}^2(0)}.
$$
Using this and putting a square root, we arrive at the first claim stated in the lemma. Keeping track of the constant, we see that the assertion about its dependence is true. For the second inequality stated in the lemma, the same reasoning works. 
\end{proof}
\begin{corollary}\label{cor:bridge}
In the assumptions of \cref{prop:Et_estimate}, given $n$, there is a constant $C$ such that, $\forall t \geq 0$
\begin{equation}
\sqrt{E^{(0,n)}}(t) \leq C(\sqrt{E^{(n,0)}}(t)+\ldots + \sqrt{E^{(1,0)}}(t) + \sqrt{E}(t)) + Cm_n(0), \label{eq:kf_from_space_to_time}
\end{equation}
and
\begin{equation}
\sqrt{E^{(n,0)}}(t) \leq C(\sqrt{E^{(0,n)}}(t)+\ldots + \sqrt{E^{(0,1)}}(t) + \sqrt{E}(t)) + Cm_n(0). \label{eq:kf_from_time_to_space}
\end{equation}
The constant $C$ depends only on a bound on $m_{n+1}(0)$ and on $n,m$.
\end{corollary}
\begin{proof}
Use the previous lemma and induction.
\end{proof}
\begin{lemma}[Energy estimates for spatial derivatives] Let $(R,z)$ be a smooth solution whose initial data differs only on a compact set from that of the background solution, and such that $m_0(0)<2R_0/3$. Then there is a constant $C$ such that, 
\begin{equation}
\sqrt{E^{(0,n)}}(t) \leq C(\sqrt{E^{(0,n)}}(0) + \ldots + \sqrt{E^{(0,1)}}(0) + \sqrt{E}(0)) + Cm_{n+1}(0), \ \forall t \geq 0.
\label{eq:Ex_estimate}
\end{equation}
The constant $C$ just depends on $n$ and on a bound on $m_{
n+1}(0)$.
\end{lemma}
\begin{proof}
First, note that we have already proved the corresponding theorem for $E^{(n,0)}$ in the \Cref{prop:Et_estimate}. Secondly, we use the \cref{eq:kf_from_space_to_time}, then the \Cref{prop:Et_estimate}, and finally the \cref{eq:kf_from_time_to_space}, this time evaluated at $t=0$. 
\end{proof}
\subsection{Returning to the original variables and proof of the main stability theorem.}
Let us define $$\mathcal{M}_k[z](t):= \norm{z}_{H^k}(t) + \norm{\partial_t z}_{H^{k-1}}(t), \ \ \ \mathcal{A}(t) := \frac{1}{2}\int_{\R} z^2 \ dx. $$
%Whenever it is clear from the context we will write $\mathcal{M}_k$ without explicit mention to $z$. We also define
Note that
$\dot{\mathcal{A}} \leq 2E^{1/2}\mathcal{A}^{1/2} \leq C(E^{1/2}(0) + m_1(0))\mathcal{A}^{1/2}
$. Therefore
$$\mathcal{A}^{1/2}(t) \leq \mathcal{A}^{1/2}(0) + C(E^{1/2}(0) + m_1(0)) t.$$
\begin{theorem}
Let $(R,z)$ be a smooth solution whose initial data differs only on a compact set from that of the background solution, and such that $m_0(0)<2R_0/3$, and let $k>0$. Then the solution is defined for all $t \geq 0$, and there is a constant $C$ such that
\begin{equation}
\mathcal{M}_k(t) \leq C\left(\mathcal{M}_k(0) + m_k(0)\right)(t+1), \ \text{for } t\geq 0.
\label{eq:Hk_estimate}
\end{equation}
The constant $C$ here depends on a bound on $m_{k}(0)$ and on $k$. 
\end{theorem}
\begin{proof}
Just note that $\mathcal{M}_k$ involves $(\partial^i_x z)^2$ with $i=0,...,k$ and $(\partial^j_x \partial_t z)^2$ with $j=0,...,k-1$, and that all of these terms appear in the t-derivative or the x-derivative term of one of the followings quantities: $\mathcal{A}, E, E^{(0,1)},\ldots, E^{(0,k-1)}$. Using this, with the above computation for $\mathcal{A}$ and the \cref{eq:Ex_estimate}, yields the \cref{eq:Hk_estimate}.
\end{proof}
From this result we can easily prove the existence of solutions with certain decay in the variables $(R,W)$. Essentially, we can return to $W$ multiplying by $e^{-t}$. The precise statement is given by the \Cref{lemma:passage_complete} in the next section.
Granted this, and recalling the definition
\begin{equation}
\widetilde{\mathcal{M}}_k(t):= \norm{W-W_b}_{\widetilde{H}_k}(t) + \norm{\partial_t(W-W_b)}_{\widetilde{H}_{k-1}}(t),
\label{eq:Mtilde}
\end{equation}
we finally obtain the main result of this section.
\begin{theorem}\label{cor:polarized_compact_sup}
Given $(R, W)$ a smooth solution to the \cref{eq:sysR,eq:sysW}, with $q=0$. Suppose also that the initial data $R, R_t, W, W_t$ differs from that of the background in a compact set and that $m_0(0)<2R_0/3$. Then, the solution is defined for every $t\geq 0$ and
$$\widetilde{\mathcal{M}}_k(t) \leq Ce^{-t}(t+1)\left(\widetilde{\mathcal{M}}_k(0) + m_k(0)\right) \ \ \forall t \geq 0.$$
Moreover, the constant $C$ depends only on a bound on $m_k(0)$ and $k$.
\end{theorem}

\section{Proof of the future long-time existence}\label{sec:future_existence}
The goal of this section is to prove a future long-time existence result for general solutions to the \cref{eq:sysR,eq:sysW,eq:sysq}. This is accomplished by the \cref{th:existencia_todo_tiempo}. Additionally, we obtain $C^0$ and $C^1$ bounds that will be useful in the next section.

\subsection{The energy-momentum tensor and a bound for the distance from the background.}
Again, let us consider a solution to the \cref{eq:sysR,eq:sysW,eq:sysq}. Note that, at first, the solutions may be defined only on a interval $[0,T)$. We assume $T$ is the maximal time of existence. In this section, we require the solutions to be $C^2$. Moreover, we ask this solutions and their time-derivatives to differ initially from the background only on a compact set. Due to a finite speed propagation argument, $R-R_b$, $W-W_b$, and $q-q_b$, where defined, are of -locally x-compact support. These properties allow us to integrate by parts and differentiate under the integral sign. 

Recall $\chi:\mathbb{K} \to \mathbb{H}$ the map in \cref{eq:chi}, and the metric $k = 4e^{4t}(-dt^2+dx^2) + R^2(t,x)d\psi^2$. Consider the energy-momentum tensor given by $T_{ab} := \partial_a\chi^i \partial_b \chi^j h_{ij} - \frac{1}{2}k_{ab}k^{\alpha \beta}h_{ij}\partial_\alpha \chi^i \partial_\beta \chi^j$. A direct computation shows that $\nabla^a T_{ab} = 0$. Consider the vector field $X := \partial_t/2e^{2t}$, the slice $S_t = \{t\}\times \R \times S^1$ and the vector field given by $T^a_{ \ b} X^b$. This vector field is the dual vector to the 1-form $T(\cdot, X)$. Using Stokes\footnote{As the solutions we are treating now differ from the background only on a compact set, we only have these two terms contributing to the flux.} we have:
$$\int_{[0,t_0]\times \R \times S^1} \nabla^a(T_{ab}X^b) \ dVol = \int_{S_0} X^aT_{ab}X^b \ dS_0 - \int_{S_t} X^aT_{ab}X^b \ dS_t $$
Now $\nabla^a(T_{ab}X^b) = T_{ab}{\nabla^aX^b}$. A computation shows that
$$
\begin{aligned}
\nabla^a X^b &= \frac{1}{4e^{6t}}(\partial_x)^a(\partial_x)^b + \frac{R_t}{R^3}\frac{1}{2e^{2t}} (\partial_\psi)^a(\partial_\psi)^b,\\
T_{00} &= T_{11} = \frac{1}{2}(\norm{\partial_t \chi}_h^2 + \norm{\partial_x \chi}_h^2),\\
T_{22} &= \frac{R^2}{2}\frac{\norm{\partial_t \chi}_h^2 - \norm{\partial_x \chi}_h^2}{4e^{4t}}.
\end{aligned}
$$
Using this, we have, 
$$
\begin{aligned}
\frac{d}{dt} \int_\R \frac{1}{2}(\norm{\partial_t \chi}_h^2 + \norm{\partial_x \chi}_h^2) \frac{R}{e^{2t}} dx &= -2\int_{\R} \frac{\norm{\partial_t \chi}_h^2}{2}\left(1 + \dfrac{R_t}{2R}\right) \frac{R}{e^{2t}} \ dx \\ &  -2\int_\R \frac{\norm{\partial_x \chi}_h^2}{2}\left(1 - \dfrac{R_t}{2R}\right) \frac{R}{e^{2t}} \ dx. \\
\end{aligned}
$$
Let us define
$$\mathcal{S} = \int_\R \frac{1}{2}(\norm{\partial_t \chi}_h^2 + \norm{\partial_x \chi}_h^2) \frac{R}{e^{2t}} dx.$$
Since $m_0(0)<\frac{2}{3}R_0$, the \cref{eq:RalphabyR_est} holds. Using this estimate yields $\mathcal{S}' \leq C(t+1)(m_1(0)+2)e^{-2t}\mathcal{S}$. 
Integrating we obtain $\mathcal{S}(t)\leq \exp(C(m_1(0)+2))\mathcal{S}(0)$ for the future. Now 
$$\begin{aligned}
\abs{W-W(-\infty)}\leq \int_\R \abs{W_x} \ dx &\leq C\sqrt{\int_\R W_x^2\cosh(2x)\ dx} \\ &\leq CS^{1/2}(t)\leq C \ exp(C(m_1(0)+2))\mathcal{S}^{1/2}(0).
\end{aligned}
$$
%$$\begin{aligned}
%\abs{W-W_b} \leq \int_{-\infty}^x \abs{(W-W_b)_x} \ dx &\leq C\sqrt{\int_\R (W-W_b)_x^2 \cosh(2x) \ dx} \\ & \leq C \mathcal{S}^{1/2}(t)\leq C \ exp(C(m_1(0)+2))\mathcal{S}^{1/2}(0),\\
%\end{aligned}$$ 
As $W(-\infty)=W_b(-\infty)$, we conclude that $W$ is bounded. We do analogously for $q-q_b = q$. This result allows us to conclude that the $C^0-$distance between the background solution and a solution $(W,q)$, whose initial data is a compactly supported perturbation of that of the background, is bounded. Geometrically, this means that if we start with a solution that is a compactly supported perturbation of a geodesic, then this perturbation evolves at a bounded distance. We have obtained
\begin{proposition}[$C^0$ estimate]\label{prop:controlC0}
Let $(R,W,q)$ be a $C^2$ solution to the system, \cref{eq:sysR,eq:sysW,eq:sysq}, such that the initial data differ from the initial data of the background solution only on a compact set. Suppose also that $m_0(0)<2R_0/3$. Then $(W,q)$ remains bounded. Moreover, this distance depends on a bound on $m_1(0)$ and on $\widetilde{\mathcal{M}}_1(0)$, where now 
$$\widetilde{\mathcal{M}}_1[W-W_b,q](t) = (\norm{W-W_b}_{\widetilde{H}^1} + \norm{(W-W_b)_t}_{\widetilde{H}^0} + \norm{q}_{\widetilde{H}^1} + \norm{q_t}_{\widetilde{H}^0})(t).$$
\end{proposition}
We define $\widetilde{\mathcal{M}}_k[\cdot,\cdot]$ in the same manner but using $k$ and $k-1$ norms. Again, whenever we say $\widetilde{\mathcal{M}}_k$, the reader should interpret $\widetilde{\mathcal{M}}_k[W-W_b,q]$. 
\subsection{$C^1$ estimate and future long-time existence.}
Again, we use the notation $\chi = (W,q)$ and $\norm{\partial_t \chi}^2 = \abs{\partial_t W}^2 + \abs{\partial_t q}^2$, and similarly for $\norm{\partial_x \chi}$. Recall also that $\norm{\partial_t \chi}_h^2 = 4\abs{\partial_t W}^2 + e^{-4W}\abs{\partial_x q}^2$, and analogously for $\norm{\partial_x \chi}_h$. 
\begin{proposition}[$C^1$ estimate]\label{prop:controlC1}
Let $(R,W,q)$ be a $C^2$ solution to the system given by \cref{eq:sysR,eq:sysW,eq:sysq}, such that the initial data differ from the initial data of the background solution only on a compact set. Suppose also that $m_0(0)<2R_0/3$. Then, there are two positive constants, $C$ and $\lambda$, such that 
$$\norm{\partial_t\chi} + \norm{\partial_x \chi} < C (\norm{\partial_t\chi}(0) + \norm{\partial_x \chi}(0))e^{\lambda t} \ \forall t \geq 0$$
Furthermore, $\lambda$ depends on a bound on $m_1(0)$, and the constant $C$ depends on a bound on $\norm{\partial_t \chi(0)}$ and $\norm{\partial_x \chi (0)}$.
\end{proposition}
\begin{proof}
Motivated by the exercise 6.9 of Tao's book, \cite{tao2006nonlinear}, we use the equations $\nabla^a T_{ab} = 0$. In coordinates, we obtain a system equivalent to
%$$\left\{\begin{array}{l}
%-\partial_t T_{00}+\partial_x T_{10}-\frac{R_t}{R} T_{00}+\frac{R_x}{R} T_{10}-\frac{R_t}{R} \frac{T_{22}}{R^2}=0 \\
%-\partial_t T_{01}+\partial_x T_{11}-\frac{R_t}{R} T_{01}+\frac{R_x}{R} \frac{T_{11}}{T_{\infty}}-\frac{R_x}{R} \frac{T_{22}}{R^2}=0
%\end{array}\right.$$
%Sumando y restando estas ecuaciones obtenemos un sistema de transporte para $T_{00}+T_{10} = \frac{1}{2}\norm{\partial_t \chi + \partial_x \chi}_h^2$ y $T_{00}-T_{10} = \frac{1}{2}\norm{\partial_t \chi-\partial_x \chi}_h^2$. El sistema queda
$$\left\{\begin{array}{l}
\frac{1}{2}\partial_{s^-}A=-\left(\frac{R_t}{R}-\frac{R_x}{R}\right) \frac{1}{2}A-\left(\frac{R_t}{R}+\frac{R_x}{R}\right) \frac{\left(\left\|\partial_{t} \chi\right\|_h^2-\left\|\partial_x \chi\right\|_h^2\right)}{2} \\
\frac{1}{2}\partial_{s^+}B=-\left(\frac{R_t}{R}+\frac{R_x}{R}\right) \frac{1}{2}B -\left(\frac{R_t}{R}-\frac{R_x}{R}\right) \frac{\left(\left\|\partial_{t} \chi\right\|_h^2-\left\|\partial_x \chi\right\|_h^2\right)}{2}
\end{array}\right.,$$
where $A = \norm{\partial_t \chi + \partial_x \chi}_h^2 $, $B = \norm{\partial_t \chi - \partial_x \chi}_h^2$, $\partial_{s^-} = \partial_t-\partial_x$, and $\partial_{s^+} = \partial_t + \partial_x$. Using that $\norm{\partial_t \chi}^2 - \norm{\partial_x \chi}^2 = B + 2 \langle \partial_t \chi - \partial_x \chi, \partial_x \chi \rangle_h$, we obtain
$$\frac{1}{2}(\partial_t + \partial_x)B \leq -\frac{R_t}{R}B + \left(\frac{R_t}{R}-\frac{R_x}{R}\right)\sqrt{B}\norm{\partial_x \chi}_h.$$ 
Due to \ref{eq:RalphabyR_est} we know that there is a constant $D>0$ such that
$$\partial_{s^+} B \leq D B + D\sqrt{B}\norm{\partial_x \chi}_h \leq D B + D \norm{\partial_x \chi}_h^2,$$ 
Additionally, $D$ depends only on a bound on $m_1(0)$. Let $M>0$ be such that $\norm{\partial_t \chi}_h,\norm{\partial_x \chi}_h < M$ at $t=0$ and let $T$ be the supreme of times $T'$ such that $\norm{\partial_t \chi}_h(t),\norm{\partial_x \chi}_h(t) < K\cdot M \ \forall t < T'$ ($K=10$). Note that $\sqrt{B}(0,\cdot) < 2M$. For $ t \in [0,T)$ we have
$$\partial_{s^+} B\leq DB + D(KM)^2,$$
and hence %$$B(s+x_0,s) \leq M^2 [(4+K^2)e^{Ds}-K^2].$$
$$B(t,\cdot) \leq M^2 [(4+K^2)e^{Dt}-K^2] \ \ \forall \ t \in [0,T].$$
Similarly,
$$A(t,\cdot) \leq M^2[(4+K^2)e^{Dt}-K^2] \ \ \forall \ t \in [0,T].$$
It follows that
%As a consequence, there is a time $T^*>0$, depending only %on $m_1(0)$, such that $\sqrt{B}(\cdot,t)<KM$ for $t\leq T^*$. The same estimate %works for $A = \norm{\partial_t \chi - \partial_x \chi}_h^2$. It %follows that 
$$\norm{\partial_t \chi}_h(t), \norm{\partial_x \chi}_h(t) \leq \frac{1}{2}(\sqrt{B}+\sqrt{A}) \leq M\sqrt{[(4+K^2)e^{Dt}-K^2]} \ \ \forall t \in [0,T].$$
This immediately implies $$T \geq \frac{1}{D}\log(\frac{2K^2}{4+K^2}) =: T^*.$$ We have proved that it takes at least $T^*$ for $\norm{\partial_t \chi}_h$ and $\norm{\partial_x \chi}_h$ to multiply its values by $K$. As $T^*$ is independent of $M$, this is exponential growth. To end, note that we can ignore the $h$ subscript in the norms. This is because, due to \Cref{prop:controlC0}, the factor $e^{-4W}$ in $h$ is bounded. Please note that this is the only place where we have used the hypothesis about the compactness of the initial data.
\end{proof}
\begin{reptheorem}{th:existencia_todo_tiempo}[Future long-time existence]\mbox{}
Consider initial data to the \cref{eq:sysR,eq:sysW,eq:sysq}, such that $R(0,\cdot)$, $W(0,\cdot)$, $q(0,\cdot)$ $\in C^2$ and $R_t(0,\cdot),W_t(0,\cdot),q_t(0,\cdot)$ $\in C^1$. Suppose additionally that $m_0(0)<2R_0/3$. Then there is a unique $C^2$ solution defined for all $t \geq 0$.
\end{reptheorem}
\begin{proof}
To tackle this problem we will use the previous estimates (\cref{prop:controlC0,prop:controlC1}). In order to use these propositions we need to have a solution such that:
\begin{enumerate}
\item It is $C^2$ and is defined in a maximal interval $[0,T)$.
\item $W-W_b$ and $q-q_b$ are of locally $x-$compact support.
\end{enumerate}

However, the solution to the given initial data will not meet the second requirement. Furthermore, at first we do not have a short time existence for this kind of initial data so we cannot guarantee the first requirement either. To circumvent this problem, we note that by a standard finite speed propagation argument it suffices to consider solutions whose initial data differs only on a compact set from the data of the background. Then we work with the functions $R-R_b$, $\Delta W = W-W_b$, and $\Delta q = q-q_b = q$. $R$ is obviously defined for all $t$. Fixed $R$, $\Delta W$ and $\Delta q$ satisfy the equations
\begin{gather}
\begin{aligned}\Delta W_{tt}-\Delta W_{xx} & + \frac{R_t}{R}\Delta W_t - \frac{R_x}{R}\Delta W_x + \frac{(\Delta q_t^2- \Delta q_x^2)}{2}e^{-4\Delta W}e^{-4W_b}  \\ &= (W_b)_x(\frac{{R}_x}{R}-\frac{{R_b}_x}{R_b}) =: r(t,x)
\end{aligned}
\label{eq:DeltaW}, \\
\Delta q_{tt}-\Delta q_{xx} + \frac{R_t}{R}\Delta q_t - \frac{R_x}{R}\Delta q_x - 4\Delta q_t\Delta W_t + 4\Delta q_x\Delta W_x + 4\Delta q_x {W_b}_x = 0 \label{eq:Deltaq},
\end{gather}

As we have compactly supported initial data for this system, it can be seen that we have $C^2$ short-time existence. For this, we use Theorem 18 from the course notes \textit{Non-linear wave equations} by Hans Ringstr\"om, with $u := (\Delta W, \Delta q)$. In fact, the theorem is stated for a source of the form $f(t,u,\partial u)$, where $u$ is the unknown, and also satisfying $f(t,0,0) = 0$. This is not our case, since the source is $r(t,x)$. However, it is easy to see that the results for 1+1-dimensional wave equations also holds for more general sources. It is enough to have a source of the form $f(t,x,u,\partial u)$, satisfying:
\begin{enumerate}[a.]
\item $f_b(t,x) := f(t,x,0,0)$ is of locally $x$-compact support.
\item For any multi-index $\alpha$, and any compact interval $I = [T_1,T_2]$, there is a function $h_{\alpha,I}:\R \to \R$, such that
$$\abs{\partial^\alpha f (t,x,u,\partial u)} \leq h_{\alpha,I}(\abs{u}+\abs{\partial u}).$$
\end{enumerate}
This indeed is the case for $r(t,x)$, and then we have $C^2$ short-time existence. Furthermore, due to the finite speed propagation property, in the interval of existence, the solutions are of locally $x$-compact support. In this way, $(W,q)$ satisfy the two requirements mentioned in the beginning of the proof, allowing us to use the previous estimates.  Consequently, fixed $T<\infty$, the $C^1$ norm of these functions remains bounded in $[0,T]$.
Using Theorem 19 from the same notes (again this is valid for this kind of source), this implies that $\norm{u}_{C^2} + \norm{\partial_t u}_{C^1}$ remains bounded in $[0,T]$. Due to Theorem 18, this implies the existence of a uniform time $\epsilon$, such that, for every $t \in [0,T]$, we can find a solution with initial data $u(t,\cdot), \partial_t u (t,\cdot)$, with existence time at least $\epsilon$. We conclude that we can extend the solution beyond any interval $[0,T]$, showing that it is defined for all $t\geq 0$.
\end{proof}

\section{Proof of the stability of compactly supported non-polarized perturbations}\label{sec:nonpol}
The purpose of this section is to study the stability for non-polarized perturbations, i.e., we allow our perturbations to have $q\neq 0$. We return to consider a smooth solution to the system formed by the \cref{eq:sysR,eq:sysW,eq:sysq}, whose initial data differs only on a compact set from the data of the background, \cref{eq:background,eq:background2,eq:background3,eq:background4}, and such that $m_0(0)<2R_0/3$. The goal is proving \Cref{cor:nonpol_comp_2}.

\subsection{A change of variables and the first order energy estimate.}
In this section, we derive the basic energy inequality from which exponential decay will follow. 
%Although is not necessary, we require $m_1(0) <\delta_1$ so $W$ and $q$ are bounded. Ultimately, we will have to ask $m_1(0)$ to be small enough so we are not worsening our hypothesis by doing this. 
Consider the change of variable 
\begin{align}
z &= R^{1/2}(W-W_b), \\
v &= \tilde{R}^{1/2} q,
\end{align} where $\tilde{R} = Re^{-4W}$. Note that in the new variables the background solution is given by $z_b = 0$, $v_b = 0$, and $R_b = R_0 \cosh(2x)$. Moreover, this change yields the following system of PDEs for $z$ and $v$:  
\begin{align}
& z_{tt} - z_{xx} + z G + B = g, \label{eq:znpol} \\
& v_{tt} -v_{xx} + v(G+4{W_b}_x^2) + D = 0,  \label{eq:v}
\end{align}
where \small
$$\left\{\begin{array}{ll}
\begin{aligned} D &= 4v(\tilde{W}_x^2 - \tilde{W}_t^2) + 8v\tilde{W}_x{W_b}_x + 2v(q_x^2 - q_t^2)e^{-4W}, \end{aligned} & B = R^{1/2}\left(\frac{q_t^2-q_x^2}{2}\right)e^{-4W}, \\
\tilde{W} = W - W_b, &  G = \frac{R_t^2-R_x^2}{4R^2},
 \\
 g = R^{1/2}\left(\frac{R_x}{R} - \frac{{R_b}_x}{R_b}\right){W_b}_x & {W_b}_x = \frac{W_0}{\cosh(2x)^2} \\
\end{array}\right.$$
\normalsize
Note that the change for $v$ is different, using $\tilde{R}$ instead of $R$. The purpose of this is to make the equation for $v$ to look structurally similar to the equation of the polarized case, \cref{eq:z}. The system displayed is exactly the system of the polarized case plus products combining $z$, $v$, and their first order derivatives. The argument here is to treat this non-linearity as a small perturbation. Also observe that the system above is coupled only through the perturbative terms $B$ and $D$. With respect to the nonlinearity, after a careful inspection, we note that the important feature about its structure is that 
\begin{equation}
B,D =  \sum_{i=1}^n R^{-m_i/2}e^{-4\widetilde{W}j_i}F_i(u), 
\label{eq:str1}
\end{equation}
where $j_i,m_i \in \mathbb{N}_{\geq 1}$, $F_i:\mathbb{R}^{12}\to \mathbb{R}$ are smooth functions, and 
\begin{equation}
\begin{aligned}
u =& (\hat{u},r), \text{with } \hat{u} = (v_x,v_t,z_x,z_t,Gv,Gz,{W_b}_xv,{W_b}_x z) \text{ and } r=(\frac{R_t}{R},\frac{R_x}{R},v,z). \\
\end{aligned} 
\label{eq:str2}
\end{equation}
Moreover, the functions $F_i$ are smooth and satisfy:
\begin{align}
F_i(0,r) = 0 \ \forall r \in \mathbb{R}^4, \label{eq:str3}\\ 
\norm{\nabla F_i (u)}\leq C(1+\norm{u}^{^k}), \ \text{for some constant } C,\label{eq:str4}
\end{align}
where $\nabla$ denotes the gradient of $F_i$, and $k$ is a positive integer. It is important to remark that we will use each of the above properties to control the effect of the nonlinearity. It is not sufficient just to consider the \cref{eq:str1,eq:str4} along with the fact that $R \sim e^{2t}\cosh(2x)$. The \cref{eq:str3} amounts to having a particular combinations of derivatives that we need to make use of it. In particular, observe that $\norm{\hat{u}}_L \leq C\mathcal{E}_0^{1/2}$, with $C$ depending on a bound on $m_1(0)$. From now on, we will sometimes redefine k, in the same way as we have been doing with the constant C.

Our goal of finding exponential decay is translated, under this change of variable, to prove that $z$ and $v$ grow at most polynomially. The precise way in which the variables $z$ and $v$ are related to $W$ and $q$ is described in the following lemma.

\begin{proposition}[Change of variable]\label{lemma:passage_complete}
Let $R$ be a solution to the \cref{eq:sysR} satisfying $m_0(0)<2R_0/3$, and consider $k>0$. Let $W,q : (a,b) \times \R \to \R$ be two functions, and define $z:= R^{1/2}(W-W_b)$ and $v := e^{2W}R^{1/2}q$. Finally, consider $t \in (a,b)$. Then, 
\begin{enumerate}
    \item $z$ is of locally $x-$compact support if and only if $W-W_b$ is of locally $x-$compact support. The same holds for $v$ and $q$.
 %   \item $z$ and $v$ are smooth if and only if $q$ and $W$ are as well.
%    \item $z,v$ are (smooth) solutions to the \cref{eq:z,eq:v} if and only if $W$ and $q$ are (smooth) solutions to the \cref{eq:sysW,eq:sysq}.
    \item $(z(t,\cdot),z_t(t,\cdot))$ belongs to $H^k\times H^{k-1}$ if and only if $((W-W_b)(t,\cdot),(W-W_b)_t(t,\cdot))$ belongs to $\widetilde{H}^k\times\widetilde{H}^{k-1}$. Furthermore,
    $$\begin{aligned}\frac{1}{C} e^{-t}(\norm{z}_{H^k} + \norm{z_t}_{H^{k-1}})(t) &\leq (\norm{W-W_b}_{\widetilde{H}^k} + \norm{(W-W_b)_t}_{\widetilde{H}^{k-1}})(t) \\ &\leq C e^{-t}(\norm{z}_{H^k} + \norm{z_t}_{H^{k-1}})(t), \end{aligned}$$
    where $\widetilde{H}^k$ is the Sobolev space given by $\|f\|_{\widetilde{H}_k}^2:=\sum_{i=0}^k \int_{\mathbb{R}}\left(f^{(i)}(x)\right)^2 \cosh (2 x) d x$, and $z_t = (R^{1/2})_t(W-W_b) + R^{1/2}(W-W_b)_t$. The constant $C$ depends only on a bound of $m_k(0)$ and $k$. 
    \item Suppose that $(z(t,\cdot),z_t(t,\cdot)) \in H^k \times H^{k-1}$, then the previous item holds for $v$ and $q$, instead of $z$ and $W-W_b$, with the difference that $C$ now also depends on $\norm{W-W_b}_{\widetilde{H}^k}(t)+\norm{{W-W_b}_t}_{\widetilde{H}^{k-1}}(t)$.
\end{enumerate}
\end{proposition}
In order to achieve the desired estimates we define the following energies,
\begin{equation}
\mathcal{A}(t) := \frac{1}{2}\int_\R z^2 + v^2 \ dx,
\end{equation}
\begin{equation}
\mathcal{E}_0(t) := E[z,v] := \frac{1}{2}\int_\R z_x^2 + z_t^2 + z^2G_b \ dx + \frac{1}{2} \int_\R v_x^2 + v_t^2 + v^2(G_b + 4{W_b}_x^2) \ dx,
\end{equation}
\begin{equation}
E_1(t) := E[z_t,v_t] + E[z,v], \  \ \mathcal{E}_1(t) := E[z_x,v_x] + E[z,v].
\end{equation}
Recall that $G_b = \frac{1}{\cosh^2(2x)} $ and ${W_b}_x^2 = \frac{W_0^2}{\cosh^2(2x)}$, so all the energies defined above are positive definite. Our goal is to prove \Cref{th:clave_no_pol}, which gives the bound
\begin{equation}
\mathcal{E}_1^{1/2} < C(\mathcal{E}_1^{1/2}(0)+ m_2(0)) \ \ \forall t \geq 0, \\ 
\end{equation} for sufficiently small $\mathcal{E}_1(0)$ and $m_2(0)$. This will be achieved by a series of lemmas. \Cref{lema:E_1_evo} establishes a bound for $\dot{E}_1$. The bound is the same appearing in the polarized case plus a sum due to the perturbative terms. The next two lemmas are intended to deal with this sum. They will provide, under certain conditions, a kind of equivalence between $E_1$ and $\mathcal{E}_1$. This will allow us to change the factor $\mathcal{E}_1^{n_i/2}$ in the sum by $E_1^{n_i/2}$. This will imply that $E_1$ is bounded, and consequently $\mathcal{E}_1$ will also be bounded, as long as $E_1(0)$ and $\mathcal{E}_1(0)$ are sufficiently small.

The outline is similar to that of the polarized case. As before, we need to use the estimates for time derivatives to get control of $\mathcal{E}_1$. The reader should think of \Cref{lema:E_1_evo} as \Cref{prop:Et_estimate}, the difference being that, due to the factor $\mathcal{E}_1$ in the right-hand side of the \cref{eq:E_1_evo}, \Cref{lema:E_1_evo} almost controls $E_1$. In the same way, the reader should think of \Cref{lemma:x_by_t,asdf} as \Cref{lemma:key_formula}, serving as a passage between time and spatial energies, this time being more subtle than before. 

%These requirements are imposed throughout all the \cref{sec:nonpol}, and, for the sake of concreteness, will be referred to as \textit{the assumptions}
\begin{lemma}\label{lema:E_1_evo}
For every solution to the system given by \cref{eq:sysR,eq:znpol,eq:v}, whose initial data differs only on a compact set from that of the background solution, and such that $m_0(0)<2R_0/3$, there is a constant $C>0$, depending only on a bound on $m_2(0)$ and on $\mathcal{E}_0(0) + \mathcal{A}(0)$, such that
\begin{align}
&\dot{\mathcal{A}} \leq 2\mathcal{A}^{1/2}\mathcal{E}_0^{1/2}, \label{eq:A_evo}\\
&\begin{aligned}
\dot{\mathcal{E}_1} & \leq  \ C(t+1)e^{-2t}E_1 + C(t+1)e^{-t}m_2(0)E_1^{1/2}\\ & + Ce^{-t}E_1(1+ \mathcal{E}_1^{k/2}+\mathcal{A}^{k/2}). \label{eq:E_1_evo}
\end{aligned}
\end{align}
\end{lemma}
\begin{proof}
Since we are working with solutions of locally $x-$compact support, we can differentiate inside the integral. Doing this with $\mathcal{A}$ and using Cauchy-Schwarz yields 
$$\dot{\mathcal{A}} = \int_\R zz_t + vv_t\ dx \leq 2\mathcal{A}^{1/2}\mathcal{E}^{1/2}_0.$$
Now let's compute $\dot{\mathcal{E}_0}$. Deriving under the integral and using parts
$$
\begin{aligned}
\dot{\mathcal{E}}_0 &= \int_\R z_t( z_{tt} - z_{xx} + z G_b) + v_t(v_{tt} - v_{xx} + v(G_b + 4{W_b}_x^2)) \ dx \\
&= \int_\R \underbrace{z_t[z(G_b-G) + g]}_{A} - z_t B + \underbrace{v_t[v(G_b -G)]}_{C} - v_tD \ dx .\\
\end{aligned}
$$
The terms $A$ and $C$ are controlled as in the polarized case:
$$\int_\R A+C \ dx \leq C(t+1)e^{-2t}\mathcal{E}_0 + C(t+1)e^{-t}m_1(0)\mathcal{E}_0^{1/2},$$
where $C$ depends on a bound on $m_1(0)$. The new thing here is to bound $\int_\R z_t B + v_t D \ dx$. Due to the \cref{eq:str1,eq:str2,eq:str3}, without lost of generality, it suffices to deal with a term of the form
$$\int_\R R^{-m/2}e^{-4\tilde{W}j}z_tF(u) \ dx \ \ \ \text{or} \ \ \ \int_\R R^{-m/2}e^{-4\tilde{W}j}v_tF(u),$$
where $m,j\geq 1$. In any case, using Cauchy-Schwarz, $R^{-m/2}\leq Ce^{-t/2}$ (\cref{eq:R_equiv}), and the fact that $\widetilde{W}$ is bounded (\cref{prop:controlC0}), these integrals are bounded by $Ce^{-t}\mathcal{E}_0^{1/2}\norm{F(u)}_{L^2}$, where $C$ depends on a bound on $m_1(0)$ and on $\mathcal{E}_0^{1/2}(0)+\mathcal{A}^{1/2}(0)$. Lastly,
$$F(u) = F(\hat{u},r) = F(\hat{u},r) - F(0,r) = \int_0^1 \nabla_{\hat{u}}F(tu)\ dt \cdot \hat{u}, $$
and then 
\begin{equation}
\abs{F(u(x,t))} \leq C(1+\norm{u}_\infty^k(t) )\norm{\hat{u}(x,t)}. \label{eq:Fnorm}
\end{equation} 
Now to control $\norm{u}_\infty$, first note that, by Sobolev embedding, $\norm{v}_\infty,\norm{z}_\infty, \norm{v_t}_\infty$,$\norm{v_x}_\infty$, \\$\norm{z_t}_\infty$,$\norm{z_x}_\infty$ $\leq C(\mathcal{A}^{1/2}+\mathcal{E}_1^{1/2})$. Secondly, $\abs{R_t/R}, \abs{R_x/R}, \abs{G}, \abs{{W_b}_x} < C$, where $C$ depends on a bound on $m_1(0)$. Lastly, note that $\norm{\hat{u}}_{L^2} \leq C\mathcal{E}_0^{1/2}$, and then 
\begin{equation}
\norm{F(u)}_{L^2} \leq C(1 + \mathcal{A}^{k/2}+\mathcal{E}_1^{k/2})\mathcal{E}_0^{1/2}, \label{eq:FL2}
\end{equation}
where $C$ depends on a bound on $m_1(0)$ and on $\mathcal{E}_0^{1/2}(0)+\mathcal{A}^{1/2}(0)$.
Putting it all together, and using $\mathcal{E}_0 \leq E_1$, we have found
\begin{equation}
\begin{aligned}
\dot{\mathcal{E}_0} & \leq  \ C(t+1)e^{-2t}E_1 + C(t+1)e^{-t}m_1(0)E_1^{1/2}\\ & + \sum_{i} Ce^{-t}E_1(1+ \mathcal{E}_1^{k/2}+\mathcal{A}^{k/2}) \leq CE_1(1+\mathcal{E}_1^{k/2}+\mathcal{A}^{k/2}),
\end{aligned}
\label{eq:first_bound}
\end{equation}
where again $C$ depends on a bound on $m_1(0)$ and on $\mathcal{E}_0^{1/2}(0)+\mathcal{A}^{1/2}(0)$.
Now, in order to obtain an estimate for $\dot{E}_1$ we will bound $\dot{E}[z_t,q_t]$. Deriving under the integral and using parts
$$
\begin{aligned}
\dot{E}[z_t,q_t] &= \int_\R z_{tt}( (z_t)_{tt} - (z_t)_{xx} + (z_t) G_b) + v_{tt}((v_t)_{tt} - (v_t)_{xx} + v_t(G_b + 4{W_b}_x^2)) \ dx \\
%&= \int_\R \underbrace{z_t[z(G_b-G) + g]}_{A} - z_t[\underbrace{R^{1/2}\left(\frac{q_t^2-q_x^2}{2}\right)e^{-4W}}_{B}] \ dx \\ &+ \int_\R \underbrace{v_t[v(G_b -G)]}_{C} - v_t\underbrace{[4v(\tilde{W}_x^2 - \tilde{W}_t^2) + 8v\tilde{W}_x{W_b}_x + 2v(q_x^2 - q_t^2)e^{-4W}]}_{D} \ dx. \\
\end{aligned}
$$
Now deriving the system, \cref{eq:znpol,eq:v}, respect to $t$ we find that 
$$\begin{aligned}
& (z_t)_{tt} - (z_t)_{xx} + z_t G + G_t z +  B' = g_t,  \\
& \begin{aligned} & (v_t)_{tt} -(v_t)_{xx} + v_t(G+4{W_b}_x^2)  + vG_t + D' = 0. \end{aligned} 
\end{aligned}$$
Therefore
$$
\begin{aligned}
\dot{E}[z_t,q_t] &= \int_\R \underbrace{z_{tt}[z_t(G_b-G) + G_t z +g_t]}_{F} - z_{tt}B' + \underbrace{v_{tt}[v_t(G_b -G)+ G_t v]}_{G} - v_{tt}D' \ dx. \\
\end{aligned}
$$
Again, the terms $F$ and $G$ are controlled as in the polarized case, yielding
$$\int_\R F+G \ dx \leq C(t+1)e^{-2t}E_1 + C(t+1)e^{-t}m_2(0)E_1^{1/2},$$
where $C$ depends on a bound on $m_2(0)$. Now we have to control $\int_\R z_{tt}B' + v_{tt}D'$. When we differentiate $B$ and $D$, we find (see the \cref{eq:str1,eq:str2,eq:str3})
\begin{equation}
\begin{aligned}
B',D' &= \sum_i \underbrace{(R^{-m_i/2})_t e^{-4\widetilde{W}j_i}F_i(u)}_{(I)} - \underbrace{4j_i R^{-m_i/2}e^{-4\widetilde{W}j_i} \widetilde{W}_t F_i(u)}_{(II)} \\ &+  R^{-m_i/2}e^{-4\widetilde{W}j_i}\underbrace{\nabla F_i(u)\cdot (\partial_t \hat{u}, \partial_t(\frac{R_t}{R}), \partial_t(\frac{R_x}{R}), \partial_t v, \partial_t z)}_{(III)}.\\
\end{aligned}
\label{eq:DifB}
\end{equation}
Note that $\abs{(R^{-m_i/2})_t} =  \abs{-m_i \tfrac{R_t}{R} R^{-m_i/2}} \leq C R^{-m_i/2}$, and that $\widetilde{W}_tF_i = R^{-1/2}(z_t - \frac{R_t}{2R}z)F_i =: R^{-1/2}H_i$, with $H_i$ satisfying the \cref{eq:str3,eq:str4}. For this reason, the terms produced  by $(I)$ and $(II)$ are controlled in the same way as above. Now, regarding $(II)$, observe that it can be split in two terms of the form 
$$\nabla_{\hat{u}}F(u)\cdot \partial_t \hat{u} + \nabla_r F(u)\cdot \partial_t r.
$$
Regarding the right term, if we call $G(\hat{u},r') := \nabla_r F(u)\cdot\partial_t r$, with $r':= (r,\partial_t r)$, then $G$ satisfies the \cref{eq:str2,eq:str3,eq:str4} but with $r'$ instead of $r$. Therefore, we can follow the same reasoning to obtain, maybe up to redefining $k$, $\norm{G(\hat{u},r')}_{L^2} \leq C(1+\mathcal{A}^{k/2}+ \mathcal{E}_1^{k/2})\mathcal{E}_0^{1/2}$. The difference is that now $C$ depends on a bound on $m_2(0)$ and on $\mathcal{A}^{1/2}(0)+\mathcal{E}_0^{1/2}(0)$. Now, let us control the left term. Observe that 
$$\abs{\nabla_{\hat{u}} F(u)\cdot\partial_t \hat{u}} \leq C\norm{\nabla_{\hat{u}} F(u)}\norm{\partial_t \hat{u}}\leq C(1+\norm{u}_\infty^{k})\norm{\partial_t \hat{u}},$$ 
and then 
$$\norm{\nabla_{\hat{u}}F(u)\cdot \partial_t \hat{u}}_{L^2}\leq C(1+\mathcal{A}^{k/2}+\mathcal{E}_1^{k/2})\norm{\partial_t \hat{u}}_{L^2}\leq C(1+\mathcal{A}^{k/2}+\mathcal{E}_1^{k/2})E_1^{1/2},$$
where we have used that $\norm{\partial_t \hat{u}}\leq CE_1^{1/2}$. With the usual bound for $R^{-m_i/2}$, we have $\abs{\int_\R z_{tt}B' + v_{tt} D' \ dx} \leq Ce^{-t/2}(1+\mathcal{A}^{K/2}+\mathcal{E}_1^{k/2})E_1$. Putting everything together,
\begin{equation}
\begin{aligned}
\dot{\mathcal{E}_1} & \leq  \ C(t+1)e^{-2t}E_1 + C(t+1)e^{-t}m_2(0)E_1^{1/2}\\ & + Ce^{-t}E_1(1+ \mathcal{E}_1^{k/2}+\mathcal{A}^{k/2}).
\end{aligned}
\label{eq:seconda_bound}
\end{equation}
\end{proof}
\begin{lemma}\label{lemma:x_by_t}
If during an interval of time $[0,T]$, we have a solution with $\abs{z_x},\abs{v_x} <1$, whose initial data differs only on a compact set from that of the background solution, and such that $m_0(0)<2R_0/3$, then, there is a constant $C>0$, depending only on a bound on $m_1(0)$ and on $\mathcal{E}_0(0)+\mathcal{A}(0)$, such that
\begin{equation}
\mathcal{E}_1 \leq CE_1 + C(t+1)^{2}e^{-2t}m_1(0)^2 + Ce^{-2t}E_1(\mathcal{A}^{k}+E_1^k) \ \text{ for } t \in [0,T].
\end{equation}
\end{lemma}
\begin{proof}
We know that $\mathcal{E}_1 = E[z,v] + E[z_x,v_x] \leq E_1 + E[z_x,v_x]$ so in order to bound $\mathcal{E}_1$ with $E_1$ we need to control $E[z_x,v_x]$ by $E_1$. Note that
$$
\begin{aligned}
E[z_x,v_x] &= \int z_{xx}^2 + z_{xt}^2 + z_{x}^2G_b + v_{xx}^2 + v_{xt}^2 + v_{x}^2(G_b + 4{W_b}_x^2)\ dx \\ & \leq \int_\R z_{xx}^2 + v_{xx}^2 + CE_1. \\
\end{aligned}
$$ 
Now, using the equation, 
$$
\begin{aligned}
& z_{xx}^2 \leq C (z_{tt}^2 + z^2G^2 + g^2 + B^2), \\
& w_{xx}^2 \leq C (w_{tt}^2 + w^2G^2 + D^2),  \\
\end{aligned}$$
so 
$$\begin{aligned}
\int_\R z_{xx}^2 + w_{xx}^2 &\leq C E_1 + \int_\R g^2 \ dx + \int_\R B^2 + D^2 \ dx \\
\leq & \ CE_1 + C(t+1)^2e^{-2t} m_1(0)^2 + \int_\R B^2 + D^2 \ dx . 
\end{aligned}
$$
So now we need to control the last integral by $E_1$. In order to do this, we will use the hypothesis that $\abs{z_x},\abs{v_x} <1$. Now, remembering the general form of $B$ and $D$ (\cref{eq:str1,eq:str2,eq:str3,eq:str4}), it follows that 
$$
\begin{aligned}
\int B^2+D^2 \ dx &\leq C \sum_i \int R^{-m_i}e^{-8\tilde{W}j_i}F^2(u) \ dx \leq C e^{-t}\sum_i \norm{F_i(u)}_{L^2}^2 \\ &\leq Ce^{-t}(1+\norm{u}_\infty^{2k})\norm{\hat{u}}_{L^2}^2 \leq Ce^{-t}(1+\norm{u}_\infty^{2k})E_1,
\end{aligned}$$
where we have grouped the summands and used the \cref{eq:Fnorm}. Now, for $\norm{u}_\infty$, for the terms with $\norm{z}_\infty,\norm{v}_\infty, \norm{z_t}_\infty, \norm{v_t}_\infty$, by Sobolev embedding, they are bounded by $C(\mathcal{A}^{1/2}+E_1^{1/2})$. Regarding the terms $\norm{v_x}_\infty,\norm{z_x}_\infty$, just use they are less than $1$ by hypothesis. Then $\norm{u}_\infty \leq C(1+\mathcal{A}^{1/2}+E_1^{1/2})$. Putting it all together
$$
\mathcal{E}_1 \leq CE_1 + C(t+1)^{2}e^{-2t}m_1(0)^2 + Ce^{-t}E_1(1+\mathcal{A}^{k}+E_1^k) \ \text{ for } t \in [0,T].$$
%so
%$$\mathcal{E}_1 \leq E_1 + D \leq CE_1 + C(t+1)^{2}e^{-2t}m_1(0)^2 + \sum_i Ce^{-t}E_1^{1+n_i/2}\mathcal{A}^{m_i/2}$$ 
\end{proof}
\begin{lemma}\label{asdf}
For every solution whose initial data differs only on a compact set from that of the background solution, and such that $m_0(0)<2R_0/3$, there is a constant $C>0$, depending only on a bound on $m_1(0)$ and $\mathcal{E}_0(0)+\mathcal{A}(0)$, such that
\begin{equation}
E_1 \leq C\mathcal{E}_1 + C(t+1)^{2}e^{-2t}m_1(0)^2 + Ce^{-2t}\mathcal{E}_1(\mathcal{A}^{k}+\mathcal{E}_1^k),
\end{equation}
in particular, evaluating at $t=0$,
\begin{equation}
E_1(0) \leq C(\mathcal{E}_1(0) + m_1^2(0)), 
\label{eq:rereAa}
\end{equation}
and now the constant also depends on a bound on $\mathcal{A}(0) + \mathcal{E}_1(0)$.
\end{lemma}
\begin{proof}
Notice that this is the inequality of the previous lemma but with $E_1$ and $\mathcal{E}_1$ reversed. Following the same argument leads to the need to control $\int B^2 + D^2 \ dx$ by $\mathcal{E}_1$. This control is achieved by the \cref{eq:FL2}.
\end{proof}
\begin{theorem}\label{th:clave_no_pol}
There is some $\delta>0$, such that, for every solution to the system, whose initial data differs only on a compact set from that of the background solution, and such that $\mathcal{E}_1^{1/2}(0),m_2(0)<\delta$, there is a constant $C$, such that 
\begin{equation}
E_1^{1/2} < C(E_1^{1/2}(0)+ m_2(0)) \ \ \forall t \geq 0,
\label{eq:Bound1}
\end{equation}
and
\begin{equation}
\mathcal{E}_1^{1/2} < C(\mathcal{E}_1^{1/2}(0)+ m_2(0)) \ \ \forall t \geq 0, \\ 
\end{equation}
Furthermore, the constant just depends on a bound on $m_2(0),\mathcal{E}_1(0)$ and on $\mathcal{A}^{1/2}(0)$.
\end{theorem}
\begin{proof}
Let $\delta'>0$ be such that $\delta'<1$ and such that if $\mathcal{E}_1^{1/2} < \delta'$ then $\abs{z_x},\abs{v_x} <1$. The existence of $\delta'$ is justified by Sobolev embedding.   Now suppose that $\mathcal{E}_1^{1/2}(0),m_2(0)<\delta << \delta'$. The value of $\delta$ will be specified in a moment. The only property that we will use now is that since $\delta < \delta'$ and $\mathcal{E}_1^{1/2}(0)<\delta$,  then $\mathcal{E}_1^{1/2}(t) < \delta'$ for at least an interval of time. Consider $$\widetilde{T} = \sup \{s: \mathcal{E}_1^{1/2}(t) < \delta' \text { for } t \in [0,s)\}.$$ For $t \in [0,\tilde{T})$, we have $\abs{z_x},\abs{v_x}<1$, and therefore we are allowed to apply \Cref{lemma:x_by_t}. This lemma asserts that
$$\mathcal{E}_1 \leq CE_1 + C(t+1)^{2}e^{-2t}m_1(0)^2 + Ce^{-t}E_1(\mathcal{A}^{k}+E_1^k) \ \text{ for } t \in [0,\tilde{T}).$$
Now by the evolution equation for $\mathcal{A}$, \cref{eq:A_evo}, we know that $$\mathcal{A}^{1/2} \leq \mathcal{A}^{1/2}(0) + \int_0^t 2\mathcal{E}_1^{1/2}(s) \ ds \leq \mathcal{A}^{1/2}(0) + 2t \ \ \ \ \ \forall t \in [0,\widetilde{T}),
$$
so $e^{-t}\mathcal{A}^{k}$ is bounded, in $[0,\widetilde{T})$, by a constant that just depends on a bound on $\mathcal{A}^{1/2}(0)$. Hence $\abs{\mathcal{E}_1}(t),\abs{m_1}(t) <1$ and $e^{-t}\mathcal{A}^{m_i}(t)$ is bounded for $t \in [0,\widetilde{T})$. Using these bounds and the previous lemma (\Cref{asdf}), we see that $E_1$ is bounded in $[0,\widetilde{T})$, and the bound depends on a bound on $m_1(0)$, $\mathcal{E}_1^{1/2}(0)$ and on $\mathcal{A}^{1/2}(0)$. As a consequence, $E_1^{1+n_i/2} < CE_1 \ \forall t \in [0,\widetilde{T})$, where again, $C$ depends on a bound on $m_1(0)$, $\mathcal{E}_1^{1/2}(0)$ and on $\mathcal{A}^{1/2}(0)$. Using this fact, we specialize the conclusion of \Cref{lemma:x_by_t}, obtaining
\begin{equation}
\mathcal{E}_1 \leq CE_1 + C(t+1)^2 e^{-2t}m_1^2(0) \ \forall t \in [0,\widetilde{T}).
\label{eq:DS}
\end{equation}
Here the constant $C$ depends on a bound on $m_1(0)$, $\mathcal{E}_1^{1/2}(0)$ and on $\mathcal{A}^{1/2}(0)$. Now, we proceed to control our energies. First, by the \cref{eq:E_1_evo} and the fact that $\mathcal{E}_1 <1$ and $\mathcal{A}^ke^{-t}< Ce^{-t/2}$ in $[0,\widetilde{T})$, we know that
$$
\begin{aligned}
\dot{E}_1 & \leq  \ C(t+1)e^{-2t}E_1 + C(t+1)e^{-t}m_2(0)E_1^{1/2} + Ce^{-t/2}E_1\ \ \ \forall t \in [0,\widetilde{T}), \\
\end{aligned}
$$
and then,
$$\dot{E}_1 \leq Ce^{-t/2}(E_1 + m_2(0)E_1^{1/2}) \ \forall t \in [0,\widetilde{T}].$$
Using Gronwall
\begin{equation}
E_1^{1/2} \leq C(E_1^{1/2}(0) + m_2(0))  \ \forall t \in [0,\widetilde{T}).
\label{eq:almost_valid_for_ever}
\end{equation}
The constant here depends on a bound on $m_2(0), \mathcal{E}_1^{1/2}
(0)$ and on $\mathcal{A}^{1/2}(0)$. Now, using the \cref{eq:DS,eq:almost_valid_for_ever} we get
$$
\begin{aligned}
\mathcal{E}_1^{1/2} &\leq C E_1^{1/2} + Cm_1(0) \leq C (E_1^{1/2}(0) + m_2(0)) \\ & \leq
C(\mathcal{E}_1^{1/2}(0)+ m_2(0)) \ \forall t \in [0,\widetilde{T}), \\
\end{aligned}
$$
where in the last inequality, we have used the \cref{eq:rereAa}. Summarizing
\begin{equation}\mathcal{E}_1^{1/2} \leq C(\mathcal{E}_1^{1/2}(0)+ m_2(0)) \ \forall t \in [0,\widetilde{T}).
\label{eq:wawa}
\end{equation}
As we are asking $\mathcal{E}_1^{1/2}(0),m_2(0)$ to be less than $\delta$, we have $\mathcal{E}_1^{1/2} \leq C 2\delta \ \forall t \in [0,\widetilde{T})$. Now, if we require $\delta < \frac{\delta'}{4C}$ then $\mathcal{E}_1^{1/2}(t) < \delta'/2 < \delta' \ \forall t \in [0,\widetilde{T})$ and hence $\widetilde{T} = +\infty$. To sum up, if $m_2(0),\mathcal{E}_1^{1/2}(0) < \delta$ then 
$$E_1^{1/2} \leq C(E_1^{1/2}(0) + m_2(0)) \ \ \forall t \geq 0$$
and since we have the \cref{eq:wawa}, then we also have
$$\mathcal{E}_1^{1/2} \leq C(\mathcal{E}_1^{1/2}(0) + m_2(0)) \ \ \forall t \geq 0.$$
\end{proof}
\subsection{Second order energy estimate and the proof of the main stability theorem.}
In order to extend our results for a more general class of functions we need to control one more derivative. For this, we have to generalize \Cref{th:clave_no_pol} to $E_2:= E_1 + E[z_{tt},v_{tt}]$ and $\mathcal{E}_2:= \mathcal{E}_1 + E[z_{xx},v_{xx}]$. The arguments here are essentially the same as above. For this reason, we only state the results, limiting ourselves to a few brief comments.
\begin{lemma}\label{lemma:E_2_evo}
For every solution to the system, \cref{eq:sysR,eq:znpol,eq:v}, whose initial data differs only on a compact set from that of the background solution, and such that $m_0(0)<2R_0/3$, there is a constant $C>0$, depending only on a bound on $m_3(0)$ and on $\mathcal{E}_0(0) + \mathcal{A}(0)$, such that
\begin{equation}
\begin{aligned}
\dot{E}_2 & \leq  \ C(t+1)e^{-2t}E_2 + C(t+1)e^{-t}m_2(0)E_2^{1/2}\\ & + Ce^{-t}E_2(1+ E_2^{k/2} + \mathcal{E}_2^{k/2}+\mathcal{A}^{k/2}). \label{eq:E_2_evo}
\end{aligned}
\end{equation}
\end{lemma}
\begin{lemma}\label{lemma:xx_by_tt}
If during an interval of time $[0,T]$ we have a solution with $\abs{z_x},\abs{v_x}, \abs{z_{xx}}$,\\ $\abs{v_{xx}} <1$, whose initial data differs only on a compact set from that of the background solution, and such that $m_0(0)<2R_0/3$, then there is a constant $C>0$, depending only on a bound on $m_2(0)$ and on $\mathcal{E}_0(0)+\mathcal{A}(0)$, such that
\begin{equation}
\mathcal{E}_2 \leq CE_2 + C(t+1)^{2}e^{-2t}m_2(0)^2(1+\mathcal{A}^{k}+E_1^k) + Ce^{-2t}E_1(1+\mathcal{A}^{2k}+E_1^{2k}) \ \ \text{for } t \in [0,T].
\end{equation}
\begin{proof}[Comments]
Doing the same computations as before leads to the need to control $\norm{\nabla B}_{L^2}^2,\norm{\nabla{D}}_{L^2}^2$. Using the \cref{eq:str1,eq:str2,eq:str3,eq:str4} to obtain a formula for $\nabla B$ and $\nabla D$, we found $$\mathcal{E}_2 \leq C(t+1)^2e^{-2t}m_2(0)^2 + Ce^{-2t}E_1(1+\mathcal{A}^k+E_1^k) + Ce^{-2t}\mathcal{E}_1(1+\mathcal{A}^k+E_1^k),$$
where, as in \Cref{lemma:x_by_t}, we have used the extra hypothesis to bound $\norm{u}_\infty$. Then, the use of \Cref{lemma:x_by_t} finishes the proof.
\end{proof}
\end{lemma}
\begin{lemma}\label{lemma:tt_by_xx}
For every solution whose initial data differs only on a compact set from that of the background solution, and such that $m_0(0)<2R_0/3$, there is a constant $C>0$, depending only on a bound on $m_2(0)$ and $\mathcal{E}_0(0)$, such that
\begin{equation}
E_2 \leq C\mathcal{E}_2 + C(t+1)^{2}e^{-2t}m_2(0)^2(1+\mathcal{A}^k + \mathcal{E}_1^k) + Ce^{-2t}\mathcal{E}_1(1+\mathcal{A}^{2k}+\mathcal{E}_1^{2k}),
\end{equation}
in particular, evaluating at $t=0$
\begin{equation}
E_2(0) \leq C(\mathcal{E}_2(0) + m_2^2(0)), 
\label{eq:rere}
\end{equation}
and now the constant also depends on a bound on $\mathcal{A}(0) + \mathcal{E}_1(0)$.
\end{lemma}
\begin{theorem}\label{th:nonpol_comp}
There is some $\delta>0$, such that, for every solution to the system, whose initial data differs only on a compact set from that of the background solution, and such that $\mathcal{E}_2^{1/2}(0),m_3(0)<\delta$, there is a constant $C$, such that 
\begin{equation}
E_2^{1/2} < C(E_2^{1/2}(0)+ m_3(0)) \ \ \forall t\geq 0,
\end{equation}
and
\begin{equation}
\mathcal{E}_2^{1/2} < C(\mathcal{E}_2^{1/2}(0)+ m_3(0)) \ \ \forall t\geq 0,
\end{equation}
for some constant $C$ that just depends on a bound on $m_3(0),\mathcal{E}_2(0)$ and on $\mathcal{A}^{1/2}(0)$.
\end{theorem}
\begin{proof}
Having the three previous lemmas, we can repeat the proof of \Cref{th:clave_no_pol}.
\end{proof}
Finally, we arrived at the main result of this section.
\begin{theorem}\label{cor:nonpol_comp_2}
There is a number $\delta>0$, such that for any smooth solution to the system, \cref{eq:sysR,eq:sysW,eq:sysq}, $(R, W,q)$, with initial data that differs from that of the background solution only on a compact set, the following holds. If $m_3(0)$ and $\widetilde{\mathcal{M}}_3(0)$ are both less than $\delta$,  
$$\widetilde{\mathcal{M}}_3(t)\leq Ce^{-t}(t+1)(\widetilde{\mathcal{M}}_3(0) + m_3(0)) \ \ \forall t \geq 0.$$
Moreover, the constant $C$ just depends on a bound on $m_3(0)$ and on $\widetilde{\mathcal{M}}_3(0)$.
\end{theorem}
\begin{proof}
The previous theorem, together with the estimate for $\mathcal{A}$, gives a bound for $\mathcal{M}_2$. Furthermore, $\mathcal{A}(0)+\mathcal{E}_2(0)$ is bounded by $\mathcal{M}_2(0)$. Now the result follows from \Cref{lemma:passage_complete}.
\end{proof}
\section{Stability for general perturbations}\label{sec:gen}
So far, we have proved results concerning compactly supported perturbations. In this section, we provide an argument to generalize these results for solutions in a larger functional space. To do this, we approximate the initial data given by a sequence of initial data which, as before, differ from that of the backgrounds solution only on a compact set. We will see that the sequence of solutions converges to the solution with the given initial data in a suitable sense. The argument is similar to the Cauchy stability but with a different functional space. Consider
$(R_1, W_1,q_1)$ and $(R_2, W_2,q_2)$, two solutions with initial data that differs from that of the background solution only on a compact set. Define $z_1 = R_1^{1/2}(W_1-W_b)$, $v_1 = R_1^{1/2}e^{-2W}q_1$ and similarly $z_2$ and $v_2$. Then, for $i=1,2$, we have
$$
\begin{aligned}
& (z_i)_{tt} - (z_i)_{xx} + f(u_i) = g_i, \\
& (v_i)_{tt} - (v_i)_{xx} + h(u_i) = 0, 
\end{aligned}
$$
where $u$ is a vector whose components are $Gv$,$Gz$,$R^{-j/2}z$,$R^{-j/2}v$, $R^{-j/2}z_x$,$R^{-j/2}z_t$, $R^{-j/2}v_x$,$R^{-j/2}v_t$, $\tfrac{R_t}{R}v$,$\tfrac{R_x}{R}v,\tfrac{R_t}{R}z$,$\tfrac{R_x}{R}z$,$z$,$z_x$,$z_t$,$v$,$v_x$,$v_t$, with $j=1,2,3$, and $f$ and $h$ are smooth functions. Additionally, we denote $\Delta z = z_2-z_1$, and similarly $\Delta v$, $\Delta f$, $\Delta g$, and $\Delta R$. Taking the difference we have
\begin{align}
& (\Delta z)_{tt} - (\Delta z)_{xx} + \Delta f = \Delta g, \label{eq:Deltaz}\\
& (\Delta v)_{tt} - (\Delta v)_{xx} + \Delta h = 0. \label{eq:Deltav}
\end{align}

In order to control the sequence, we need to introduce the following energy
$$\mathcal{H}_n \! = \! \! \frac{1}{2}\! \sum_{i=0}^n \! \int_\R [(\partial_x^i \Delta z)_x]^2 + [(\partial_x^i \Delta z)_t]^2 + (\partial_x^i \Delta z) ^2 + [(\partial_x^i \Delta v)_x]^2 + [(\partial_x^i \Delta v)_t]^2 + (\partial_x^i \Delta v)^2  dx.$$
Again, some estimates are required to bound $\mathcal{H}_n$. For convenience, we introduce the notation
$$m_k[f](t) := \norm{f}_{C^k}(t) + \norm{\partial_t f}_{C^k}(t), \ \ \ m_k(t) := \sup_{i=1,2} m_k[R_i](t). $$
Also note that we now have two definitions for $\mathcal{E}_n$, namely, one for $(z_1,v_1)$ and one for $(z_2,v_2)$. Now, we will call $\mathcal{E}_n$ to the maximum of these two.

\begin{lemma}[Control of $\mathcal{H}_n$]\label{prop:controlHn}
Let $(R_1,z_1,v_1)$ and $(R_2,z_2,v_2)$ be two solutions to the system given by the \cref{eq:sysR,eq:znpol,eq:v}, such that their initial data at $t=0$ differ from that of the background solution only on a compact set. Suppose additionally that $m_0(0)<2R_0/3$. Then, there is a continuous function $\kappa:\R^2 \to \R^+$ such that
\begin{equation}
\dot{\mathcal{H}}_n \leq [\kappa(\norm{u_1}_{H^n},\norm{u_2}_{H^n})](\mathcal{H}_n + e^{-2t}(t+1)^2m_{n+1}^2[\Delta R])(0)) \label{eq:controlcauchy}
\end{equation}
Furthermore, $\kappa$ only depends on a bound on $m_{n+1}(0)$.
\end{lemma}
\begin{proof}
This follows from standard hyperbolic estimates.
\end{proof}

Now we can enhance the results given by  \Cref{cor:polarized_compact_sup,cor:nonpol_comp_2}. 
\begin{reptheorem}{th:polarized_perturbations}[Stability for polarized perturbations]
Let $k\geq 3$ and assume $m_0(0)<2R_0/3$. Then, for any initial data to the \cref{eq:sysR,eq:sysW,eq:sysq}, satisfying $(R-R_b)(0,\cdot) \in C_0^k$, $\partial_t(R-R_b)(0,\cdot) \in C_0^{k-1}$, $(W-W_b)(0,\cdot) \in \widetilde{H}^k$, $\partial_t(W-W_b)(0,\cdot) \in \widetilde{H}^{k-1}$, and $q(0,\cdot) = q_t(0,\cdot) = 0$, there is a unique $C^2$ solution defined for all time $t \geq 0$ and
$$\widetilde{\mathcal{M}}_k(t) \leq Ce^{-t}(t+1)\left(\widetilde{\mathcal{M}}_k(0) + m_k(0)\right).$$
Moreover, the constant $C$ depends only on an upper bound on $m_k(0)$ and $k$.
\end{reptheorem}
\begin{proof}
Recall that we have already proved the existence for all time $t\geq 0$. We proceed in steps. 

\noindent
1. Setting: consider a sequence of initial data such that, $(R_i-R_b)(0,\cdot)$ and $({R_{i}}_t-{R_b}_t)(0,\cdot)$ are compactly supported and converge to $((R-R_b)(0,\cdot),(R-R_b)_t(0,\cdot))$ in $C_0^k \times C_0^{k-1}$. Similarly, we ask the same for the initial data of $W$, but this time, we require that the sequence $(z_i(0,\cdot),{z_i}_t(0,\cdot))$ converges to $(z(0,\cdot),z_t(0,\cdot))$ in $H^k \times H^{k-1}$. 
In what follows, we will sometimes use $v_i$, even though $v_i \equiv {v_i}_t \equiv 0$. In this way the argument presented here is general, and a slight modification applies to the non-polarized case. 

\noindent
2. Rough control: using the \cref{eq:Hk_estimate} we have
$$\norm{(z_i,v_i)}_{H^k}(t)+\norm{({z_i}_t,{v_i}_t)}_{H^{k-1}}(t) \leq  C(T+1) \leq C_T \ \forall t \in [0,T].$$

Note that we can choose the same constant $C_T$ for each $i$. Using \Cref{lemma:coef_est_2}, it is easy to see that  
$$\norm{u_i}_{H^{k-1}}(t) \leq C(\norm{(z_i,v_i)}_{H^k}(t)+\norm{({z_i}_t,{v_i}_t)}_{H^{k-1}}(t) )$$
where $C$ depends only on a bound on $m_{k}[R_i](0)$, and hence we can choose the same constant for each $i$. It follows that 
$$\norm{u_i}_{H^{k-1}}(t) \leq C(\norm{(z_i,v_i)}_{H^k}(t)+\norm{({z_i}_t,{v_i}_t)}_{H^{k-1}}(t) )\leq C_T \ \forall t \in [0,T], \ \forall i$$

\noindent
3. Cauchy Sequence: A quick computation using D'Alembert shows that for each fixed $T>0$, $(R_i,{R_i}_t)$ is a Cauchy sequence in $C([0,T],C_0^k)\times C([0,T], C_0^{k-1})$. Regarding $z_i$, the rough control obtained above allows us to apply the \cref{eq:controlcauchy} with $n=k-1$, but changing $\kappa$ with a constant independent of $i$. Integrating that inequality
$$\mathcal{H}_{k-1}(t) \leq C(\mathcal{H}_{k-1}(0) + m_{k}^2[\Delta R](0)) \ \ \forall t \in [0,T], \forall i.$$

Consequently, $(z_i,{z_i}_t)$ is a Cauchy sequence in $C([0,T],H^k)\times C([0,T],H^{k-1})$. By Sobolev embedding, $(z_i,{z_i}_t)$ is a Cauchy sequence $C([0,T],C^2)\times C([0,T],C^1)$ ($k\geq 3$). Using the equation we obtain 
$$\abs{\partial_t^2 \Delta z} \leq \abs{\partial_x^2 \Delta z} + Cm_{1}[\Delta R](0) + C(m_{2}[\Delta R](0) + \norm{\Delta z}_{H^2} + \norm{\Delta z_t}_{H^1}).$$ We have proved that, for each fixed $T>0$, $z_i$ is a Cauchy sequence in $C^2([0,T]\times \R,\R)$. The limit function is the $C^2$ solution of the initial data. We have also proved that $(z,z_t)$ is a Cauchy sequence in $C([0,T],H^k)\times C([0,T],H^{k-1})$.

\noindent 
4. Exponential decay: 
Now, by the \cref{eq:Hk_estimate}, we have
$$\mathcal{M}_k[z_i](t)\leq C(t+1)\left(\mathcal{M}_k[z_i](0)+m_k[R_i](0)\right) \ \forall t \geq 0,$$
where we have used the same constant for each $i$. Taking limit in the above estimate yields
$$
\mathcal{M}_k[z](t)\leq C(t+1)\left(\mathcal{M}_k[z](0)+m_k[R](0)\right) \ \forall t \geq 0.
$$
By the change of variable lemma, \Cref{lemma:passage_complete}, this solution $z$ corresponds to the smooth solution $W$ given in the statement. The estimate in the statement is a consequence of this lemma as well.
\end{proof}

\begin{reptheorem}{th:nonpol_perturbations}[Stability for non-polarized perturbations]
Assume $m_0(0) < 2R_0/3$. Then, for any initial data to the \cref{eq:sysR,eq:sysW,eq:sysq}, satisfying $(R-R_b)(0,\cdot) \in C_0^3$, $\partial_t(R-R_b)(0,\cdot) \in C_0^{2}$, $(W-W_b)(0,\cdot) \in \widetilde{H}^3$, $\partial_t(W-W_b)(0,\cdot) \in \widetilde{H}^{2}$, $q(0,\cdot) \in \widetilde{H}^3$, and $q_t(0,\cdot)$ in $\widetilde{H}^2$, there is a unique $C^2$ solution defined for all time $t \geq 0$. Furthermore, there exists $\delta>0$, independent of the solution, such that if $m_3(0)<\delta$, $\widetilde{\mathcal{M}}_3(0)<\delta$, then
$$\widetilde{\mathcal{M}}_3(t)\leq Ce^{-t}(t+1)(\widetilde{\mathcal{M}}_3(0) + m_3(0)). $$
Moreover, the constant $C$ depends only on an upper bound on $\delta$.
\end{reptheorem}
\begin{proof}
The argument is essentially the same, the difference being that now we work with $z_i$ and $v_i$, that it is \Cref{th:nonpol_comp} which allows us to use \Cref{prop:controlHn} with $n=2$, and also that we use the equation given by \Cref{cor:nonpol_comp_2}, version $\mathcal{M}_3$, instead of the \cref{eq:Hk_estimate}.
\end{proof}
%\noindent\textbf{Declarations} \mbox{}\\
%\textbf{Data availability statement} Data sharing not applicable to this article as no datasets were generated or analysed during the current study. \mbox{}\\
%\textbf{Conflict of interest} The authors have no relevant financial or non-financial interests to disclose.
\bibliography{referencias} 

\begin{thebibliography}{10}

\bibitem{anderson2001long}
Michael~T Anderson.
\newblock On long-time evolution in general relativityand geometrization of
  3-manifolds.
\newblock {\em Communications in Mathematical Physics}, 222(3):533--567, 2001.

\bibitem{anderson2004cheeger}
Michael~T Anderson.
\newblock Cheeger-gromov theory and applications to general relativity.
\newblock {\em The Einstein Equations and the Large Scale Behavior of
  Gravitational Fields: 50 Years of the Cauchy Problem in General Relativity},
  pages 347--377, 2004.

\bibitem{chrusciel1990space}
Piotr~T Chru{\'s}ciel.
\newblock On space-times with u (1)$\times$ u (1) symmetric compact cauchy
  surfaces.
\newblock {\em Annals of Physics}, 202(1):100--150, 1990.

\bibitem{fischer2000reduced}
Arthur~E Fischer and Vincent Moncrief.
\newblock The reduced hamiltonian of general relativity and the
  $\sigma$-constant of conformal geometry.
\newblock {\em Mathematical and Quantum Aspects of Relativity and Cosmology},
  537:70, 2000.

\bibitem{fischer2001reduced}
Arthur~E Fischer and Vincent Moncrief.
\newblock The reduced einstein equations and the conformal volume collapse of
  3-manifolds.
\newblock {\em Classical and Quantum Gravity}, 18(21):4493, 2001.

\bibitem{gowdy1974vacuum}
Robert~H Gowdy.
\newblock Vacuum spacetimes with two-parameter spacelike isometry groups and
  compact invariant hypersurfaces: Topologies and boundary conditions.
\newblock {\em Annals of Physics}, 83(1):203--241, 1974.

\bibitem{reiris2010ground}
Martin Reiris.
\newblock The ground state and the long-time evolution in the {CMC} {Einstein}
  flow.
\newblock In {\em Annales Henri Poincar{\'e}}, volume~10, pages 1559--1604.
  Springer, 2010.

\bibitem{ringstrom2004wave}
Hans Ringstr{\"o}m.
\newblock On a wave map equation arising in general relativity.
\newblock {\em Communications on Pure and Applied Mathematics: A Journal Issued
  by the Courant Institute of Mathematical Sciences}, 57(5):657--703, 2004.

\bibitem{ringstrom2004gowdy}
Hans Ringstr{\"o}m.
\newblock On gowdy vacuum spacetimes.
\newblock In {\em Mathematical Proceedings of the Cambridge Philosophical
  Society}, volume 136, pages 485--512. Cambridge University Press, 2004.

\bibitem{ringstrom2009cauchy}
Hans Ringstr{\"o}m.
\newblock {\em The Cauchy problem in general relativity}, volume~6.
\newblock European Mathematical Society, Helsinki, 2009.

\bibitem{ringstrom2009strong}
Hans Ringstr{\"o}m.
\newblock Strong cosmic censorship in {T}$^3$-gowdy spacetimes.
\newblock {\em Annals of mathematics}, pages 1181--1240, 2009.

\bibitem{sideris1989global}
Thomas~C Sideris.
\newblock Global existence of harmonic maps in minkowski space.
\newblock {\em Communications on pure and applied mathematics}, 42(1):1--13,
  1989.

\bibitem{tao2006nonlinear}
Terence Tao.
\newblock {\em Nonlinear dispersive equations: local and global analysis}.
\newblock Number 106. American Mathematical Soc., Providence, 2006.

\end{thebibliography}
%\end{document}

\end{document}